\documentclass[11pt,reqno]{amsart}
\usepackage[T1]{fontenc}
\usepackage[utf8]{inputenc}
\usepackage{lmodern}
\usepackage{float}
\usepackage[margin=1.2in]{geometry}
\usepackage{microtype}
\usepackage{mathtools,amssymb,amsthm,mathrsfs}
\usepackage{bm}
\usepackage{commath}
\usepackage{tikz}
\usetikzlibrary{shapes.geometric}
\usepackage{ytableau}
\usepackage[colorlinks=true,
linkcolor=blue!50!black,
citecolor=green!50!black,
urlcolor=blue!50!black]{hyperref}
\usepackage[nameinlink,noabbrev,capitalize]{cleveref}
\usepackage[
backend=biber,
url=false,
isbn=false,
doi=true,
backref=true,
style=alphabetic,
maxbibnames=99,
sorting=nyt,
giveninits=true
]{biblatex}
\DefineBibliographyStrings{english}{
	backrefpage = {$\uparrow$},   
	backrefpages = {$\uparrow$}, 
}
\numberwithin{equation}{section}
\theoremstyle{plain}

\newtheorem*{theorem*}{Theorem}
\newtheorem{prop}{Proposition}
\newtheorem{lemma}{Lemma}

\newcommand{\thistheoremname}{}
\newtheorem{genericthm}{\thistheoremname}

\newtheorem*{genericthm*}{\thistheoremname}
\newenvironment{namedtheorem*}[1]{\renewcommand{\thistheoremname}{#1}\begin{genericthm*}}{\end{genericthm*}}

\newtheorem{maintheorem}{Theorem}

\crefname{maintheorem}{Theorem}{Theorems}
\newtheorem{theoremp}{Theorem}

\crefname{theoremp}{Theorem}{Theorems}
\crefname{appendix}{Appendix}{Appendices}

\theoremstyle{definition}

\newtheorem{example}{Example}

\Crefname{prop}{Proposition}{Propositions}

\allowdisplaybreaks

\DeclareMathOperator{\ad}{ad}
\DeclareMathOperator{\CT}{CT}
\newcommand{\hyper}[3][F]{{}_{#2}#1_{#3}}
\let\L\relax\DeclareMathOperator{\L}{\mathcal{L}}
\DeclareMathOperator{\R}{\mathcal{R}}
\DeclareMathOperator{\M}{\mathcal{M}}
\DeclareMathOperator{\N}{\mathcal{N}}
\def\geq{\geqslant}
\def\leq{\leqslant}
\def\({\left(}
\def\){\right)}
\DeclareMathOperator{\cover}{:\!\supset}
\DeclareMathOperator{\coveredby}{\subset\!:}
\newcommand{\poch}[2]{\(#1\)_{{#2}}}
\let\=\relax\newcommand{\=}{\mathrel{\phantom{=}}}
\newcommand\Z{\mathbb{Z}}
\newcommand\Q{\mathbb{Q}}

\newcommand*{\SetSuchThat}[1][]{} 

\newcommand*{\MvertSets}{%
	\renewcommand*\SetSuchThat[1][]{%
		\mathclose{}%
		\nonscript\;##1\vert\penalty\relpenalty\nonscript\;%
		\mathopen{}%
	}%
}
\MvertSets 
\DeclarePairedDelimiterX \Set [2] {\lbrace}{\rbrace}
{\,#1\SetSuchThat[\delimsize]#2\,}

\newcommand{\mydef}[1]{\textbf{#1}}
\title{Differential Equations for Macdonald's Hypergeometric Series}
\author{Hong Chen}
\address[H.~Chen]{Department of Mathematics, University of California, Los Angeles, 520 Portola Plaza, Los Angeles, CA 90095, USA}
\email{hchen@math.ucla.edu}

\author{Siddhartha Sahi}
\address[S.~Sahi]{Department of Mathematics, Rutgers University, 110 Frelinghuysen Rd, Piscataway, NJ 08854, USA}
\email{sahi@math.rutgers.edu}

\date{\today}

\begin{document}
\begin{abstract}

	We study differential equations for certain multivariable hypergeometric series, introduced by Macdonald in the 1980s, which generalize hypergeometric functions of matrix argument due to Bochner and Herz.
	Such equations are well studied in the one-variable case, where they serve to characterize the classical hypergeometric functions $_p F_q(z)$ and provide a route to many of their key properties.
	Accordingly, Macdonald asked for a similar characterization for the multivariate series.
	This problem, however, has remained open.
	Differential equations are known only for a handful of cases with  $p\leq 3$ and $q\leq 2$.

	In this paper we give a complete solution to Macdonald's problem. In the multivariable case there are \emph{two} series $_p F_q(x)$ and $_p F_q(x,y)$, depending on $p+q+1$ parameters and one or two sets of variables, $x=(x_1,\dots,x_n)$ and  $y=(y_1,\dots,y_n)$. We give differential equations characterizing both series, for all values of $p$ and $q$. The proof involves several new ideas in the theory of symmetric functions, especially Jack polynomials and interpolation polynomials.

\end{abstract}
\maketitle
\section{Introduction}

We study differential equations satisfied by certain higher-rank (multivariable) hypergeometric series, which were introduced by I.~G. Macdonald \cite{Mac-HG} in the 1980s as a common, 1-parameter, generalization of the matrix hypergeometric functions for Hermitian and symmetric matrices due to S.~Bochner and C.~Herz, respectively.
Such differential equations have been intensively studied in rank 1, where they serve to characterize the classical (1-variable) hypergeometric functions $_p F_q(z)$  and provide a route to many of their key properties.
Doubtless motivated by this, Macdonald posed the problem of finding an analogous characterization  for the higher-rank series.
Since then, however, the problem has remained open.
Differential equations have been found in only a handful of cases with $p \leq 3$ and $q \leq 2$.

In this paper we provide a complete solution to Macdonald's problem.
To describe our results we note that in higher rank there are in fact \emph{two} series $_p F_q(x)$ and $_p F_q(x,y)$, which depend on one or two sets of variables, $x=(x_1,\ldots,x_n)$ and  $y=(y_1,\ldots,y_n)$, along with $p+q+1$ additional parameters.
In \cref{thm:A} we give a differential equation for $_p F_q(x,y)$, and in \cref{thm:B,thm:C} we give two different differential equations for $_p F_q(x)$.
Each equation characterizes the corresponding series subject to certain conditions.
However, \cref{thm:C} requires fewer conditions than \cref{thm:B}.
Also, while \cref{thm:A,thm:B} specialize to known results for small values of $p$ and $q$, \cref{thm:C} seems to be new even for these cases.

The rest of the introduction is organized as follows.
In \cref{sec:intro-rank-one,sec:intro-higher-rank} we recall some background on the classical and higher-rank hypergeometric series, and then we state our results precisely in \cref{sec:intro-main-results}.
We next summarize previously known results in \cref{sec:intro-previous-results}, and then give a brief description of our methods in \cref{sec:intro-techniques}, which involve several new ideas in the theory of symmetric functions, especially Jack polynomials and interpolation polynomials.
Finally, in \cref{sec:intro-organization} we sketch the organization of the rest of the paper.
\medskip

\subsection{The rank-1 theory}\label{sec:intro-rank-one}
The study of hypergeometric functions originated with Euler in 1769.
He investigated the following second-order linear ordinary differential equation (ODE),
\begin{align}\label{eqn:euler}
	z(1-z)\dod[2]{F(z)}{z} +(c-(a+b+1)z)\dod{F(z)}{z} -abF(z)=0,
\end{align}
gave a power series solution, and established an integral representation of this series.

The solution of \cref{eqn:euler} was later studied extensively by Gauss in 1812 and is now known as the \mydef{Gauss hypergeometric function} $\hyper{2}{1}(a,b;c;z)$.
For $c\notin\{0,-1,-2,\dots\}$, it is the unique solution of \cref{eqn:euler} that is analytic at $z=0$ and satisfies $F(0)=1$.

The Gauss hypergeometric function is
\begin{align}\label{eqn:2F1}
	\hyper{2}{1}(a,b;c;z)
	=\sum_{k=0}^\infty \frac{\poch{a}{k}\poch{b}{k}}{\poch{c}{k}} \frac{z^k}{k!},
\end{align}
where $\poch{a}{k}=a(a+1)\cdots(a+k-1)$ is the \mydef{Pochhammer symbol}.
The series converges for $|z|<1$ provided that $c\neq0,-1,-2,\dots$.
It includes, as special or limiting cases, the power, logarithm, and exponential functions, as well as orthogonal polynomials such as the Jacobi polynomials.

\Cref{eqn:euler} was subsequently studied by Kummer (1836), most notably, by Riemann (1857), and by Schwarz (1873) and Klein (1894).
Let $p(z)$ and $q(z)$ be meromorphic functions.
We say that the second-order linear ODE
\begin{align}\label{eqn:ode}
	\dod[2]{F(z)}{z} +p(z)\dod{F(z)}{z}+q(z)F(z)=0,
\end{align}
has a regular singularity at $z=z_0$ if at least one of $p$ and $q$ is singular there, while $(z-z_0)p(z)$ and $(z-z_0)^2q(z)$ are analytic at $z_0$.
Riemann showed that any equation of the form \eqref{eqn:ode} with exactly three singular points, all regular, can be transformed into Euler's hypergeometric differential equation \eqref{eqn:euler}, with singularities at $0$, $1$, and $\infty$.

More generally, a series $\displaystyle\sum_{k=0}^\infty c_k$ is said to be \mydef{hypergeometric} if the ratio $c_{k+1}/c_k$ is a rational function of $k$.
Many elementary and non-elementary functions in mathematics and physics are hypergeometric functions.
It is customary to let $c_0=1$ and write
\begin{align*}
	\frac{c_{k+1}}{c_k}=\frac{(k+a_1)\cdots(k+a_p)}{(k+b_1)\cdots(k+b_q)}\frac{z}{k+1},
\end{align*}
which leads to
\begin{align}\label{eqn:pFq}
	\hyper{p}{q}(\underline{a};\underline{b};z) = \sum_{k=0}^\infty \frac{\poch{a_1}{k}\cdots \poch{a_p}{k}}{\poch{b_1}{k}\cdots \poch{b_q}{k}} \frac{z^k}{k!},
\end{align}
where $\underline{a}=(a_1,\dots,a_p)$ and $\underline{b}=(b_1,\dots,b_q)$.
As a formal power series with constant term $1$, $\hyper{p}{q}$ is the unique solution, for $b_1,\dots,b_q\notin\{0,-1,-2,\dots\}$, of the differential equation
\begin{align}\label{eqn:pfq-de}
	\(z\dod{}{z}\prod_{k=1}^q \(z\dod{}{z}+b_k-1\)
	-z \prod_{k=1}^p \(z\dod{}{z}+a_k\)\)(F(z))=0.
\end{align}
For details, see \cite{Erd81}, the Bateman Manuscript Project.

Hypergeometric differential equations are canonical examples in the Fuchsian theory, which studies linear ODEs with singularities using the Frobenius series method.
A central problem in Fuchsian theory is to understand the monodromy of solutions, i.e., how solutions transform under analytic continuation along loops encircling singularities.
The monodromy of the hypergeometric differential equation for $\hyper{n}{n-1}$ has been studied in detail in \cite{BH89}.

Besides differential equations, central topics in the theory of hypergeometric functions include integral representations (Euler, Barnes, etc.), summation formulas (Gauss, Pfaff--Saalschütz, Dougall, etc.), asymptotics, and transformations (Euler, Pfaff, Kummer, etc.).
We refer to \cite{AAR99} for a comprehensive survey of special functions, especially hypergeometric functions.

Basic hypergeometric series (i.e., $q$-analog) are discussed briefly in \cref{sec:basic}.

\subsection{The higher-rank theory}\label{sec:intro-higher-rank}
The discussion above concerns hypergeometric functions in a single complex variable~$z$.
Since Herz \cite{Herz} introduced hypergeometric functions of matrix arguments, many other functions have been extended to the symmetric cone of real symmetric positive-definite matrices.
These functions are foundational in multivariate statistics (e.g., Wishart and Hotelling distributions, \cite{Muirhead82}), random matrix theory (eigenvalue distributions and Tracy--Widom laws, \cite{Johnstone}), representation theory on symmetric cones (spherical functions and Dunkl operators, \cite{Dunkl,Rosler03}), and computational mathematics \cite{KE06}.

\subsubsection{Zonal polynomials and hypergeometric functions}
Herz \cite{Herz} used recursive applications of the Laplace transform and its inverse to define hypergeometric functions $\hyper{p}{q}$ with matrix arguments.
Later, Constantine \cite{Con63} gave a series expansion for them in terms of \mydef{zonal polynomials}, which were introduced by James and Constantine in the 1960s.

Zonal polynomials take as input an $n\times n$ symmetric positive-definite matrix $X$.
Equivalently, they can be viewed as symmetric polynomials in the eigenvalues $(x_1,\dots,x_n)$ of $X$.
Like other families of symmetric polynomials, $n$-variate zonal polynomials are indexed by \mydef{partitions} of length at most~$n$.
Such a partition is a tuple $\lambda=(\lambda_1,\dots,\lambda_n)\in\Z^n$ with $\lambda_1\geq\cdots\geq\lambda_n\geq0$.

Constantine \cite{Con63} introduced the following hypergeometric functions associated with zonal polynomials, or simply \mydef{zonal hypergeometric functions}, with one and two matrix arguments $x=(x_1,\dots,x_n)$ and $y=(y_1,\dots,y_n)$
\begin{align}
	\hyper{p}{q}(\underline{a};\underline{b};x)
	&=	\sum_{\lambda} \frac{\poch{\underline{a}}{\lambda}}{\poch{\underline{b}}{\lambda}} \frac{C_\lambda(x)}{|\lambda|!},	\label{eqn:pFq-Constantine-1}
	\\
	\hyper{p}{q}(\underline{a};\underline{b};x,y)
	&=	\sum_{\lambda} \frac{\poch{\underline{a}}{\lambda}}{\poch{\underline{b}}{\lambda}} \frac{C_\lambda(x)C_\lambda(y)}{|\lambda|!C_\lambda(\bm1_n)},
	\label{eqn:pFq-Constantine-2}
\end{align}
where $\poch{\cdot}{\lambda}$ is a generalization of the usual Pochhammer symbol $\poch{\cdot}{k}$, $C_\lambda$ denotes the zonal polynomial (in the $C$-form), $|\lambda|=\lambda_1+\dots+\lambda_n$ is the size of $\lambda$, and $\bm1_n=(1,\dots,1)$ ($n$ times).
The full definitions are given in \cref{sec:pre}.

The properties of zonal polynomials and zonal hypergeometric functions, and their applications in statistics, together with historical accounts, are collected in \cite{Muirhead82}.
See also the discussion below.
Beyond statistics, zonal polynomials also appear in representation theory, where they serve as spherical functions of the Gelfand pair $(GL_n(\mathbb R),O_n)$ \cite[Chapter~VII]{Mac15}.

\subsubsection{Jack polynomials and hypergeometric series}
Around 1970, Jack \cite{Jack} defined a family of symmetric polynomials, now known as \mydef{Jack polynomials} $J_\lambda(x;\alpha)$.
These depend on a parameter $\alpha$ and reduce to Schur polynomials and zonal polynomials when $\alpha=1$ and $\alpha=2$, respectively, up to normalization.
They are eigenfunctions for a second-order operator $\square$ known as the \mydef{Laplace--Beltrami operator}.
More generally, they are simultaneous eigenfunctions of a commuting family of differential operators $D_1,D_2,\dots,D_n$ known as the \mydef{Debiard--Sekiguchi operators}.
See \cite{KS06} for the historical background on Jack polynomials and related symmetric functions such as the Hall--Littlewood and Macdonald polynomials.

In a widely circulated manuscript from the late 1980s, now available as \cite{Mac-HG}, Macdonald introduced further generalizations of hypergeometric series in the alphabets $x=(x_1,\dots,x_n)$ and $y=(y_1,\dots,y_n)$:
\begin{align}
	\hyper{p}{q}(\underline{a};\underline{b};x;\alpha)
	&=	\sum_\lambda \frac{\poch{\underline{a};\alpha}{\lambda}}{\poch{\underline{b};\alpha}{\lambda}} \alpha^{|\lambda|} \frac{J_\lambda(x;\alpha)}{j_\lambda},	\label{eqn:pFq-Mac-1}\\
	\hyper{p}{q}(\underline{a};\underline{b};x,y;\alpha)
	&=	\sum_\lambda
	\frac{\poch{\underline{a};\alpha}{\lambda}}{\poch{\underline{b};\alpha}{\lambda}} \alpha^{|\lambda|} \frac{J_\lambda(x;\alpha)J_\lambda(y;\alpha)}{j_\lambda J_\lambda(\bm1_n;\alpha)},\label{eqn:pFq-Mac-2}
\end{align}
where $\poch{\cdot;\alpha}{\lambda}$ and $j_\lambda$ are defined in \cref{sec:pre}.
Macdonald implicitly treated the series in \cref{eqn:pFq-Mac-1,eqn:pFq-Mac-2} as formal power series, and we will adopt this convention.
We call these the hypergeometric series associated with Jack polynomials, or simply \mydef{Jack hypergeometric series}, in one or two alphabets.

It should be noted that Kor\'anyi \cite{Koranyi} independently introduced the one-alphabet case \eqref{eqn:pFq-Mac-1}, and Yan \cite{Yan92} the two-alphabet case \eqref{eqn:pFq-Mac-2}.
As generalizations of the zonal case, when $\alpha=2$, \cref{eqn:pFq-Mac-1,eqn:pFq-Mac-2} specialize to \cref{eqn:pFq-Constantine-1,eqn:pFq-Constantine-2}, respectively.

Many classical properties of univariate hypergeometric functions have been extended to the multivariate setting---for example, Laplace transforms \cite{Ros20,BR23}, integral representations and Selberg integrals \cite{Yan92,FW08}.
In addition, multivariate Hermite, Laguerre, and Jacobi polynomials (which correspond to the special cases $\hyper{2}{0}$, $\hyper{1}{1}$, and $\hyper{2}{1}$, respectively) have been studied \cite{Mac-HG,Lassalle-Hermite,Lassalle-Laguerre,Lassalle-Jacobi,BF97,Ros98}.

\subsubsection{Hypergeometric functions associated with root systems}
Hypergeometric functions associated with root systems generalize the Gauss hypergeometric function $\hyper{2}{1}(z)$ to multivariate settings in a different direction, governed by finite reflection groups.
The foundational theory was developed by Heckman--Opdam in the late 1980s \cite{HO87}.
In 1993, Beerends--Opdam \cite{BO93}, building on the differential equations in \cite{Yan92}, proved that Macdonald's $\hyper{2}{1}(x)$ is a special case of the hypergeometric function associated with the root system $BC_n$.
See also \cite[Chapter~8]{AskeyBateman}, the Askey--Bateman project.

\subsubsection{Applications of multivariate hypergeometric series}
With the notation in place, we briefly discuss applications in statistics.
The classical real Wishart ensembles are probability distributions over symmetric positive-definite matrices.
They play a central role in likelihood-ratio tests \cite{Muirhead82}, multidimensional Bayesian analysis \cite{KK10}, and random matrix theory \cite{James60}.
They are closely related to zonal polynomials, which are Jack polynomials with parameter $\beta=1$, where $\beta\coloneqq 2/\alpha$.
The complex and quaternion Wishart ensembles correspond to $\beta=2$ (Schur polynomials) and $\beta=4$ (quaternionic zonal polynomials), respectively.

In \cite{betaWishart}, the $\beta$-Wishart ensembles were introduced.
The parameter $\beta$ interpolates between the real ($\beta=1$), complex ($\beta=2$), and quaternionic ($\beta=4$) Wishart ensembles, enabling a unified perspective.
The eigenvalue distributions of $\beta$-Wishart matrices were studied in \cite{EK14}.
Jack polynomials and Jack hypergeometric series are key tools in such generalizations.

In \cite{BO05}, Borodin and Olshanski studied the $z$ measures on partitions and showed that the averages of the products of certain characteristic polynomials are given by a Jack series $\hyper[\widehat{F}]{2}{1}$.
This series is different from Macdonald's $\hyper{2}{1}$, although for $n=1$ both series reduce to Gauss's classical $\hyper{2}{1}$.
In \cref{appendix}, we give differential equations characterizing $\hyper[\widehat{F}]{2}{1}$.
\subsection{Main results}\label{sec:intro-main-results}
We treat $\alpha,a_1,\dots,a_p,b_1,\dots,b_q$ as algebraically independent indeterminates.
The coefficient field for the hypergeometric series is $\mathbb K=\Q(\alpha,a_1,\dots,a_p,b_1,\dots,b_q)$.
All uniqueness statements below are over this field.
Numerical specializations require the denominators and the recursion coefficients used in the proofs to remain nonzero.

In this paper we construct differential equations that characterize $\hyper{p}{q}(x)$ and $\hyper{p}{q}(x,y)$.
More precisely, we construct four partial differential operators:
\begin{align*}
	\L=\hyper[\L]{}{q}^{(x)}(\underline{b}), \quad
	\M=\hyper[\M]{p}{}^{(x)}(\underline{a}), \quad
	\N=\hyper[\N]{}{q}^{(x)}(\underline{b}), \quad
	\R=\hyper[\R]{p}{}^{(x)}(\underline{a}).
\end{align*}
The operator $\L$ \textit{lowers} degree by one, $\R$ \textit{raises} it by one, while $\M$ and $\N$ \textit{preserve} degree and act diagonally on Jack polynomials $J_\lambda(x)$.

Our lowering and raising operators, $\L$ and $\R$, admit a concise expression, using iterated commutators of the Laplace--Beltrami operator $\square$ with the differential operator $E_1=\sum_i \frac{\partial}{\partial x_i}$ and the multiplication operator $e_1=\sum_i x_i$, respectively.
Our eigen-operators, $\M$ and $\N$, involve combinations of the Debiard--Sekiguchi operators $D_r$.
The ingredients are recalled in \cref{sec:pre}.
The operators $\L$, $\M$, $\N$, and $\R$ are defined explicitly in \cref{sec:2arg,sec:1arg} (\eqref{eqn:lower-x}, \eqref{eqn:pD1}, \eqref{eqn:Dq1}, and \eqref{eqn:raise-x}, respectively).
We now define
\begin{align*}
	\mathcal A^{(x,y)} = \L^{(x)}-\R^{(y)}, \quad
	\mathcal B^{(x)}= \L^{(x)}-\M^{(x)}, \quad \mathcal C^{(x)}= \N^{(x)}-\R^{(x)}.
\end{align*}
We also write $\mathscr{F}^{(x)}$ and $\mathscr{F}^{(x,y)}_D$ for the spaces consisting of power series of the following forms:
\begin{align*}
	\mathscr{F}^{(x)}=\big\{\sum_\lambda C_\lambda J_\lambda(x)\big\}, \quad
	\mathscr{F}^{(x,y)}_D=\big\{\sum_\lambda C_\lambda J_\lambda(x)J_\lambda(y)\big\}.
\end{align*}
Write $\bm0_n=(0,\dots,0)$ ($n$ times).

\begin{namedtheorem*}{Theorem~A} The series $\hyper{p}{q}(x,y)$ is the unique solution in $\mathscr{F}^{(x,y)}_D$ of the differential equation
	\begin{align}\label{eqn:A}
		\mathcal A^{(x,y)}(F(x,y)) =0, \quad F(\bm0_n,\bm0_n)=1.
	\end{align}
\end{namedtheorem*}
\begin{namedtheorem*}{Theorem~B} The series $\hyper{p}{q}(x)$ is the unique solution in $\mathscr{F}^{(x)}$ of the differential equation
	\begin{align}\label{eqn:B}
		\mathcal B(F(x))=0,\quad F(\bm0_n)=1
	\end{align}
	subject to a certain stability condition (see \cref{eqn:pDq1x(F)-m} below).
\end{namedtheorem*}
\begin{namedtheorem*}{Theorem~C}  The series $\hyper{p}{q}(x)$ is the unique solution in $\mathscr{F}^{(x)}$ of the differential equation
	\begin{align}\label{eqn:C}
		\mathcal C(F(x))=0, \quad F(\bm0_n)=1.
	\end{align}
\end{namedtheorem*}
The techniques of this paper can also be applied to Borodin and Olshanski's $\hyper[\widehat{F}]{2}{1}$.
In \cref{appendix}, we give analogs of \cref{thm:B,thm:C} for $\hyper[\widehat{F}]{2}{1}$.

\subsection{Previous results}\label{sec:intro-previous-results}

Special cases of \cref{thm:A,thm:B} were previously known in the literature for a few values of $(p,q)$.
For the general Jack case, these are the five pairs, $(0,0)$, $(1,0)$, $(0,1)$, $(1,1)$ and $(2,1)$, marked by red squares in \cref{fig}.
In the zonal case, results were also known for the two additional pairs $(2,2)$ and $(3,2)$, marked by blue dots in \cref{fig}.

\begin{figure}[H]
	\centering
	\begin{tikzpicture}[scale=1]
		\draw[step=1cm,gray!40,very thin] (0,0) grid (4,4);
		\draw[thick] (0,0) -- (4.2,0) node[right] {$p$};
		\draw[thick] (0,0) -- (0,4.2) node[above] {$q$};

		\foreach \x in {0,1,2,3,4}
		\draw (\x,0.08) -- (\x,-0.08) node[below,yshift=-2pt] {\small $\x$};
		\foreach \y in {0,1,2,3,4}
		\draw (0.08,\y) -- (-0.08,\y) node[left,xshift=-2pt] {\small $\y$};

		\node[circle, draw, fill=blue, inner sep=3pt] at (2,2) {};
		\node[circle, draw, fill=blue, inner sep=3pt] at (3,2) {};
		\node[rectangle, draw, fill=red, inner sep=3pt] at (2,1) {};
		\node[rectangle, draw, fill=red, inner sep=3pt] at (1,0) {};
		\node[rectangle, draw, fill=red, inner sep=3pt] at (1,1) {};
		\node[rectangle, draw, fill=red, inner sep=3pt] at (0,0) {};
		\node[rectangle, draw, fill=red, inner sep=3pt] at (0,1) {};
	\end{tikzpicture}
	\caption{Previously known differential equations for $\hyper{p}{q}$}
	\label{fig}
\end{figure}

For the five red squares in the Jack case, Macdonald \cite{Mac-HG} constructed the operators $\mathcal L$, $\mathcal R$, and $\mathcal M$.
He conjectured \cref{eqn:A} and verified \cref{eqn:B}, but did not address the issue of uniqueness.
In the crucial case $(p,q)=(2,1)$, \cref{thm:B} was proved independently by Yan \cite{Yan92} and Kaneko \cite{Kan93}, and \cref{thm:A} by Baker--Forrester \cite{BF97}.

In the zonal case, \cref{thm:B,thm:A} were proved earlier for the five red squares by Constantine--Muirhead \cite{CM72} and Muirhead \cite{M70}, respectively, and for the two blue dots by Fujikoshi \cite{Fujikoshi}.

\subsection{Techniques}\label{sec:intro-techniques}
We now briefly discuss the techniques used in our work.

\subsubsection{\texorpdfstring{\cref{thm:A}}{Theorem A}} The existing proofs of \cref{thm:A}, for the small number of cases in \cref{fig} discussed above, do not readily extend to larger $p$ and $q$.

The key insight of this paper was twofold.
The first was the realization that all existing differential operators could be written as the difference of two one-alphabet operators, $\mathcal A^{(x,y)}= \L^{(x)}-\R^{(y)}$.
The second realization was that $\L$ and $\R$ could be expressed as explicit combinations of iterated commutators of the Laplace--Beltrami operator $\square$ with the differential operator $E_1=\sum_i \frac{\partial}{\partial x_i}$ and the multiplication operator $e_1(y)=\sum_i y_i$, respectively.

It turns out that this approach generalizes to all $p,q$ and leads to a proof of \cref{thm:A}, which has two further ingredients, namely the Okounkov--Olshanski binomial theorem for Jack polynomials \cite{OO97} and some results of Lassalle \cite{Las90,Las98} for Macdonald polynomials.

\subsubsection{\texorpdfstring{\cref{thm:B}}{Theorem B}}
Muirhead \cite{M70} and Fujikoshi \cite{Fujikoshi} treated the zonal cases for small $p$ and $q$.
In the Jack case, Macdonald \cite{Mac-HG}, Yan \cite{Yan92}, and Kaneko \cite{Kan93} generalized Muirhead's operator for $\hyper{2}{1}$ in a natural way (replacing $2$ by $\alpha$ in appropriate places).
This operator $\mathcal B=\L-\M$ has two parts: $\L$ lowers the degree of Jack polynomials, while $\M$ acts as an eigen-operator.

For $p=2$, Macdonald constructed eigen-operators in \cite[p.~47--49, Lemmas~2 and~3]{Mac-HG}, with eigenvalues dictated by the structure of $\hyper{p}{q}$.
The construction used the action of commutators among the differential operator $E_1$, the Euler operator $E_2$, and the Laplace--Beltrami operator $\square$ on $\hyper{0}{0} = \exp(e_1)$.
For $p=3$, we computed higher commutators and their actions on $\hyper{0}{0}$ and found that the Debiard--Sekiguchi operators are also required.
Despite the considerable computational complexity, we succeeded in constructing the eigen-operator $\M$ for $p=3$, thereby generalizing Fujikoshi's operator in the zonal case.
In principle, one could try to use this approach to find eigen-operators for higher $p$ in a step-by-step manner, but in practice, the computations soon grow intractable.

Our actual approach departs significantly from Macdonald's.
Rather than constructing each eigen-operator with prescribed eigenvalue \textit{individually}, we obtain a closed-form formula for the generating function of \textit{all} such eigenvalues.
This, in turn, yields a generating function of eigen-operators expressed in terms of all Debiard--Sekiguchi operators.

In proving uniqueness in \cref{thm:B}, we encounter difficulties also present in \cite{M70,Fujikoshi,Yan92,Kan93}: the recursion relation \eqref{eqn:rec-up} proceeds in the wrong direction, producing an under-determined system.
In previous results, it was assumed that for each $n$, the $n$-variate hypergeometric series is annihilated by the $n$-variate differential operator.
This enabled them to double the number of equations, thus obtaining an over-determined system.
With some effort, it can be shown that this system has a unique solution.
Instead, we introduce a stability assumption on the hypergeometric series: for fixed $n$ and each $m\leq n$, the $m$-variate differential operator annihilates the $m$-variate hypergeometric series.
With this alternative assumption, we can also prove the uniqueness.
The necessity of such an extra assumption is illustrated by \cref{ex:stability}.

\subsubsection{\texorpdfstring{\cref{thm:C}}{Theorem C}}
The non-uniqueness in \cref{thm:B} stems from the use of the lowering operator $\L$, which produces the problematic recursion \eqref{eqn:rec-up}.
Here we introduce a \emph{new idea}: to use the raising operator $\R$ instead, and to construct a corresponding eigen-operator $\N$ by a similar generating-function approach.
This idea leads directly to a simple recursion \eqref{eqn:C-rec-1}, and the uniqueness then follows readily, as shown in \cref{thm:C}.
\subsection{Organization}\label{sec:intro-organization}
The paper is organized as follows:

In \cref{sec:pre}, we recall standard definitions and facts concerning partitions and symmetric polynomials, and introduce the notation used throughout the paper.
Our notation largely follows \cite{Mac15,Mac-HG}, with minor modifications.

In \cref{sec:2arg,sec:1arg}, we first address the two-alphabet case and then the one-alphabet case.
Each section begins with some examples to illustrate the core ideas.
We then define the differential operators and state and prove the corresponding theorems.

In \cref{sec:future}, we discuss directions for future work, including the ideals $\hyper[\mathcal D]{p}{q}$ of annihilating differential operators for $\hyper{p}{q}$, $q$-analogs and analogs involving Macdonald polynomials, as well as non-symmetric analogs.

In \cref{appendix}, we apply the technique of the paper to Borodin and Olshanski's hypergeometric function $\hyper[\widehat{F}]{2}{1}$ and obtain differential equation characterizations of the function.

Parts of this work also appear in the first author's doctoral thesis \cite{ChenThesis2026}.
\section{Preliminaries}\label{sec:pre}
Throughout let $n\geq1$ be the number of variables and let $x=(x_1,\dots,x_n)$ and $y=(y_1,\dots,y_n)$.
We denote $\bm1_n=(1,\dots,1)$ and $\bm0_n=(0,\dots,0)$ ($n$ times).
\subsection{Partitions and Pochhammer symbols}
A \mydef{partition} is a tuple $\lambda=(\lambda_1,\lambda_2,\dots)$ of non-negative integers in weakly decreasing order $\lambda_1\geq\lambda_2\geq\cdots$, such that only finitely many are non-zero.
The \mydef{length} $\ell(\lambda)$ and the \mydef{size} $|\lambda|$ are the number and sum of these non-zero parts:
\begin{align}
	\ell(\lambda)\coloneqq\max\Set{i}{\lambda_i>0}, \quad |\lambda|\coloneqq\lambda_1+\dots+\lambda_{\ell(\lambda)}
\end{align}
We denote by $(0)$ the zero partition and set $\ell((0))=0$.

We identify a partition with its \mydef{Young diagram}, namely,  $$\Set{(i,j)\in\Z^2}{1\leq j\leq \lambda_i,\,1\leq i\leq n}.$$
The \mydef{conjugate} of $\lambda$, denoted $\lambda'$, is the partition whose Young diagram is the transpose of that of $\lambda$.
Equivalently,
\begin{align*}
	\lambda_j'=\#\Set{i}{\lambda_i \geq j}, \qquad j \geq 1.
\end{align*}

Throughout we restrict to partitions of length at most $n$ and denote the set of such partitions by $\mathcal P_n$.
We write $\mathcal P_n'$ for the set of their conjugates, i.e., partitions of arbitrary length whose largest part $\lambda_1$ is at most $n$.

For $(i,j)\in\lambda$, the \mydef{arm}, \mydef{co-arm}, \mydef{leg}, and \mydef{co-leg} are
\begin{align}
	a_\lambda(i,j) \coloneqq \lambda_i-j,\quad a'_\lambda(i,j) \coloneqq j-1,\quad
	l_\lambda(i,j) \coloneqq \lambda_j'-i,\quad l_\lambda'(i,j) \coloneqq i-1.
\end{align}
For example, if $\lambda=(7,7,6,4,4,2,1)$, then $\lambda'=(7,6,5,5,3,3,2)$, and the box $s=(3,2)$ has $a_\lambda(s)=4$, $a_\lambda'(s)=1$, $l_\lambda(s)=3$, and $l_\lambda'(s)=2$ (see \cref{fig:young-diagram}).

\begin{figure}[H]
	\centering
	\scalebox{0.8}{\begin{ytableau}
		\phantom{1}&l'&\phantom{1}&\phantom{1}&\phantom{1}&\phantom{1}&\phantom{1}\\
		\phantom{1}&l'&\phantom{1}&\phantom{1}&\phantom{1}&\phantom{1}&\phantom{1}\\
		a'&s&a&a&a&a\\
		\phantom{1}&l&\phantom{1}&\phantom{1}\\
		\phantom{1}&l&\phantom{1}&\phantom{1}\\
		\phantom{1}&l\\
		\phantom{1}
	\end{ytableau}}
	\caption{Arm, co-arm, leg, and co-leg of the box $s=(3,2)$ in the Young diagram of $(7,7,6,4,4,2,1)$.}
	\label{fig:young-diagram}
\end{figure}

We use the following partial orders on $\mathcal P_n$.
We say $\lambda$ \mydef{contains} $\mu$, written $\lambda\supseteq\mu$, if $\lambda_i\geq\mu_i$ for $1\leq i\leq n$.
We say $\lambda$ \mydef{covers} $\mu$, written $\lambda\cover\mu$, if $\lambda\supseteq\mu$ and $|\lambda|=|\mu|+1$.
We say $\lambda$ \mydef{dominates} $\mu$, written $\lambda\geq\mu$, if $|\lambda|=|\mu|$ and $\lambda_1+\dots+\lambda_i\geq\mu_1+\dots+\mu_i$ for $1\leq i\leq n$.

The usual \mydef{Pochhammer symbol} is
\begin{align}
	\poch{a}{m}\coloneqq\frac{\Gamma(a+m)}{\Gamma(a)}=a(a+1)\cdots(a+m-1),	\quad m\geq0.
\end{align}
Define the \mydef{$\alpha$-Pochhammer symbol} by
\begin{align}\label{eqn:alpha-poch}
	\poch{a;\alpha}{\lambda}
	\coloneqq \prod_{(i,j)\in\lambda}\(a+a_\lambda'(i,j)-\frac{l_\lambda'(i,j)}{\alpha}\)
	=	\prod_{i=1}^n \poch{a-\frac{i-1}{\alpha}}{\lambda_i},
\end{align}
where the statistic $a_\lambda'(i,j)-\frac{l_\lambda'(i,j)}{\alpha}$ is called the \mydef{$\alpha$-content}.
If $\lambda=(m)$ is a row partition, then $\poch{a;\alpha}{\lambda}=\poch{a}{m}$.
For a tuple $\underline{a}=(a_1,\dots,a_p)$, let
\begin{align}
	\poch{\underline{a};\alpha}{\lambda}\coloneqq \poch{a_1;\alpha}{\lambda}\cdots\poch{a_p;\alpha}{\lambda}.
\end{align}
The Pochhammer symbol used in \cref{eqn:pFq-Constantine-1,eqn:pFq-Constantine-2} is $\poch{a}{\lambda}=\poch{a;\alpha=2}{\lambda}$.
From now on, we use the general parameter $\alpha$ and write simply $\poch{a}{\lambda}=\poch{a;\alpha}{\lambda}$ and $\poch{\underline{a}}{\lambda}=\poch{\underline{a};\alpha}{\lambda}$.

Consider the sum of $\alpha$-contents
\begin{align}
	\rho(\lambda) &\coloneqq \sum_{(i,j)\in\lambda}\(a_\lambda'(i,j)-\frac{l_\lambda'(i,j)}{\alpha}\) = \sum_{i=1}^n \(\frac{\lambda_i(\lambda_i-1)}{2}-\frac{\lambda_i(i-1)}{\alpha}\).
\end{align}
If $\lambda\cover\mu$, and $(i,j)$ is the added box, then we write
\begin{align}
	\rho(\lambda/\mu) \coloneqq \rho(\lambda)-\rho(\mu) = a_\lambda'(i,j)-\frac{l_\lambda'(i,j)}{\alpha}.
\end{align}
In this case, we have the following crucial relation
\begin{align}\label{eqn:poch-ratio}
	\frac{\poch{a}{\lambda}}{\poch{a}{\mu}} = a+\rho(\lambda/\mu).
\end{align}
For a parameter tuple $\underline a=(a_1,\dots,a_p)$, we use the corresponding adjacent ratio
\begin{align*}
	\frac{\poch{\underline a}{\lambda}}{\poch{\underline a}{\mu}}
	=\prod_{k=1}^p\bigl(a_k+\rho(\lambda/\mu)\bigr),
	\qquad \lambda\cover\mu.
\end{align*}
These adjacent ratios are polynomials in the parameters and are interpreted after cancellation at special parameter values.

\subsection{Symmetric polynomials}
Let $\Lambda_{n,\Q}$ be the algebra of symmetric polynomials in $n$ variables over $\Q$.
For $\eta=(\eta_1,\dots,\eta_n)\in\mathbb Z_{\geq0}^n$, let $x^\eta \coloneqq x_1^{\eta_1}\cdots x_n^{\eta_n}$.
The \mydef{monomial symmetric polynomial} is
\begin{align}
	m_\lambda(x) \coloneqq \sum_{\eta} x^\eta
\end{align}
where $\eta$ runs over all distinct permutations of $\lambda$.
Then $(m_\lambda)_{\lambda\in\mathcal P_n}$ forms a natural basis of $\Lambda_{n,\Q}$.

The $r$-th \mydef{elementary symmetric polynomial} $e_r$ ($1\leq r\leq n$) and \mydef{power sum} $p_r$ ($r\geq1$) are
\begin{align}
	e_r(x) \coloneqq m_{(1^r)}(x) = \sum_{1\leq i_1<\dots<i_r\leq n}x_{i_1}\dots x_{i_r},	\text{\quad and \quad}
	p_r(x) \coloneqq m_{(r)}(x) = \sum_{1\leq i\leq n} x_i^r.
\end{align}
For a partition $\lambda=(\lambda_1,\dots,\lambda_l)$, define $e_\lambda \coloneqq e_{\lambda_1}\cdots e_{\lambda_l}$ and $p_\lambda \coloneqq p_{\lambda_1}\cdots p_{\lambda_l}$, where $l=\ell(\lambda)$.
Then $(e_\lambda)_{\lambda\in\mathcal P_n'}$ and $(p_\lambda)_{\lambda\in\mathcal P_n'}$ also form bases for $\Lambda_{n,\Q}$.
The \mydef{complete homogeneous symmetric polynomial} of degree $r\geq0$ is $h_r(x)\coloneqq\sum_{\lambda\in\mathcal P_n,\ |\lambda|=r}m_\lambda(x)$.
We use $e_0=h_0=1$, $e_r=0$ for $r<0$ or $r>n$, and $h_r=0$ for $r<0$.

The \mydef{Schur polynomial} $s_\lambda$ is
\begin{align}
	s_\lambda(x) \coloneqq \frac{\det(x_i^{n-j+\lambda_j})_{1\leq i,j\leq n}}{\det(x_i^{n-j})_{1\leq i,j\leq n}},
\end{align}
where the denominator $\det(x_i^{n-j})_{1\leq i,j\leq n}$ is the Vandermonde determinant
\begin{align}
	V(x)\coloneqq \prod_{i<j} (x_i-x_j).
\end{align}
Then $(s_\lambda)_{\lambda\in\mathcal P_n}$ also forms a basis for $\Lambda_{n,\Q}$.

\subsection{Jack polynomials}
Let $\Lambda_{n,\Q(\alpha)}=\Lambda_{n,\Q}\otimes\Q(\alpha)$, where $\alpha$ is an indeterminate over $\Q$.
To define the Hall inner product and the integral normalization, first work in the algebra $\Lambda_{\Q(\alpha)}$ of symmetric functions in infinitely many variables.
The \mydef{$\alpha$-Hall inner product} on this algebra is defined by
\begin{align}
	\langle p_\lambda,p_\mu\rangle_{\alpha} \coloneqq \delta_{\lambda\mu} z_\lambda \alpha^{\ell(\lambda)},
\end{align}
where $z_\lambda\coloneqq\prod_r (r^{m_r}m_r!)$ and $m_r\coloneqq \#\Set{i}{\lambda_i=r}$ is the multiplicity of $r$ in $\lambda$.

Here $\lambda$ and $\mu$ range over all partitions.
It is well known (\cite{St89} and \cite[Chapter~VI]{Mac15}) that the \mydef{integral form Jack symmetric function} $J_\lambda=J_\lambda(\cdot;\alpha)$ is uniquely determined by orthogonality, triangularity and normalization:
\begin{gather}
	\langle J_\lambda,J_\mu\rangle_{\alpha} = 0,\quad	\lambda\neq\mu,	\\
	J_\lambda=\sum_{\mu\leq\lambda}\tilde K_{\lambda\mu}m_\mu,	\\
	\tilde K_{\lambda,(1^{|\lambda|})}=|\lambda|!.
\end{gather}
Specializing $x_{n+1}=x_{n+2}=\cdots=0$ gives the $n$-variable Jack polynomial, which is nonzero when $\ell(\lambda)\leq n$.
We use $J_\lambda$ for both the symmetric function and its specialization when the meaning is clear.
The coefficients $\tilde K_{\lambda\mu}=\tilde K_{\lambda\mu}(\alpha)$ are generalizations of the Kostka numbers $K_{\lambda\mu}$ for Schur polynomials.
Knop--Sahi \cite{KSinv} gave a combinatorial interpretation for $\tilde K_{\lambda\mu}(\alpha)$ and showed that they are polynomials with non-negative integer coefficients in $\alpha$.

Jack polynomials specialize to several familiar families.
At $\alpha=0,\,1$ and $\infty$, $J_\lambda(\cdot;\alpha)$ becomes $e_{\lambda'}$, $s_\lambda$ and $m_\lambda $ (up to normalization), respectively.
In addition, when $\alpha=2$, $J_\lambda(\cdot;\alpha)$ is the zonal polynomial, see \cite[Chapter~VII]{Mac15} and \cite[Chapter~7]{Muirhead82}.

Jack polynomials are \textbf{stable} in the sense that $m$-variate ($m\leq n$) Jack polynomials are specializations of $n$-variate ones:
\begin{align}\label{eqn:stable}
	J_\lambda(x_1,\dots,x_{m};\alpha) = J_\lambda(x_1,\dots,x_m,0,\dots,0;\alpha),
	\quad  \ell(\lambda)\leq m.
\end{align}

Define the \mydef{$\alpha$-hook-length} functions:
\begin{align}
	c_\lambda(i,j) &\coloneqq a_\lambda(i,j)\alpha+l_\lambda(i,j)+1=(\lambda_i-j)\alpha+\lambda_j'-i+1,
	\\c_\lambda'(i,j) &\coloneqq (a_\lambda(i,j)+1)\alpha+l_\lambda(i,j)= (\lambda_i-j+1)\alpha+\lambda_j'-i,
\end{align}
and set
\begin{align}
	c_\lambda\coloneqq \prod_{(i,j)\in\lambda} c_\lambda(i,j), \quad
	c_\lambda'\coloneqq \prod_{(i,j)\in\lambda} c_\lambda'(i,j),\quad
	j_\lambda\coloneqq c_\lambda c_\lambda'.
\end{align}

In the symmetric-function algebra, these satisfy $\langle J_\lambda,J_\lambda\rangle_\alpha=j_\lambda$.
The \mydef{dual form} Jack polynomial is
\begin{align}
	J_\lambda^* \coloneqq \frac{J_\lambda}{j_\lambda }.
\end{align}
In statistics (e.g., \cite{Yan92,betaWishart}), one often uses the \mydef{$C$-form}, normalized by
\begin{align}
	(x_1+\dots+x_n)^d = \sum_{|\lambda|=d}C_\lambda^{(\beta)}(x).
\end{align}
This form relates to the dual form via
\begin{align}
	C_\lambda^{(\beta)}(x) = (2/\beta)^{|\lambda|} \cdot |\lambda|! \cdot  J_\lambda^*(x;\alpha=2/\beta).
\end{align}
The zonal polynomial in \cref{eqn:pFq-Constantine-1,eqn:pFq-Constantine-2} and in \cite{CM72,Fujikoshi,Muirhead82} is $C_\lambda^{(\beta=1)}$.
From now on, the symbol $C_\lambda$ denotes coefficients.

By \cite[VI.(10.25)]{Mac15}, the evaluation of $J_\lambda$ at $\bm1_n=(1,\dots,1)$ is
\begin{align}
	J_\lambda(\bm1_n;\alpha)  =\prod_{(i,j)\in\lambda}((j-1)\alpha+n-i+1) = \alpha^{|\lambda|} \poch{n/\alpha}{\lambda}.
\end{align}
The \mydef{unital form} is
\begin{align}
	\Omega_\lambda(x;\alpha) \coloneqq \frac{J_\lambda(x;\alpha)}{J_\lambda(\bm1_n;\alpha)},
\end{align}
so that $\Omega_\lambda(\bm1_n;\alpha)=1$.
For adjacent partitions $\lambda\cover\mu$ and $m\geq\max\{1,\ell(\mu)\}$, the dual evaluations satisfy
\begin{align*}
	\frac{J_\lambda^*(\bm1_m)}{J_\mu^*(\bm1_m)}
	=\frac{J_\lambda(\bm1_m)}{j_\lambda}\frac{j_\mu}{J_\mu(\bm1_m)}
	=\frac{j_\mu}{j_\lambda}\bigl(m+\alpha\rho(\lambda/\mu)\bigr).
\end{align*}
We use the dual-evaluation ratio to shorten normalization factors below.

\subsection{Binomial formula}
The following \mydef{binomial formula for Jack polynomials} is fundamental:
\begin{align}\label{eqn:bino}
	\Omega_{\lambda}(x+t\bm1_n;\alpha) = \sum_{\mu\subseteq\lambda}t^{|\lambda|-|\mu|}\binom{\lambda}{\mu}\Omega_\mu(x;\alpha),
\end{align}
where $t\bm1_n=(t,\dots,t)$ and $\binom{\lambda}{\mu}$ is the \mydef{generalized binomial coefficient}.

These coefficients were first studied in the zonal case in statistics \cite{Bi74,Muirhead82}, and in the Schur case, Lascoux \cite{Lascoux} used them to compute the Chern classes of the exterior and symmetric squares of a vector bundle (see also \cite[p.~47, Example~10]{Mac15}).
In the 1990s, they were studied by Lassalle \cite{Las90,Las98}, Kaneko \cite{Kan93}, Okounkov--Olshanski \cite{OO-schur,OO97}.
The generalized binomial coefficients can be given by evaluation of \mydef{interpolation Jack polynomials}, first studied in \cite{Sahi94,KS96,OO97}.
Recently, the $BC$-symmetry analog of the binomial coefficients and interpolation polynomials were studied in \cite{CS24}.

Although \cref{eqn:bino} depends on the number of variables $n$, the generalized binomial coefficients $\binom{\lambda}{\mu}$ \textit{do not} depend on $n$.
They admit explicit formulas in the adjacent case $\lambda\cover\mu$, and the general $\binom{\lambda}{\mu}$ can be obtained recursively via a weighted-sum formula.
See \cite{Kan93,Las98,Sahi-Jack,CS24}.
\begin{prop}\label{prop:bino}
	Let $\lambda\cover\mu$ with $\lambda/\mu=(i_0,\lambda_{i_0})$, and let $R$ and $C$ be the set of other boxes in the row and column of $\lambda/\mu$, respectively, that is,
	\begin{align}\label{eqn:RC}
		R=\Set{(i_0,j)}{1\leq j\leq \lambda_{i_0}-1},\quad C=\Set{(i,\lambda_{i_0})}{1\leq i\leq i_0-1}.
	\end{align}
	Then
	\begin{align}\label{eqn:bino-comb}
		\binom{\lambda}{\mu} = \prod_{s\in C}\frac{c_\lambda(s)}{c_\mu(s)} \prod_{s\in R}\frac{c_\lambda'(s)}{c_\mu'(s)}.
	\end{align}
\end{prop}

\subsection{Jack hypergeometric series}
We now define the $n$-variate \mydef{hypergeometric series associated with Jack polynomials}, or \mydef{Jack hypergeometric series}:
\begin{align}
	\hyper{p}{q}(x)
	= \hyper{p}{q}^{(n)}(\underline{a};\underline{b};x;\alpha)
	&\coloneqq
	\sum_{\lambda\in\mathcal P_n} \frac{\poch{\underline{a}}{\lambda}}{\poch{\underline{b}}{\lambda}} \alpha^{|\lambda|} J_\lambda^*(x;\alpha),	\label{eqn:pFq-Mac-1-new}\\
	\hyper{p}{q}(x,y)
	= \hyper{p}{q}^{(n)}(\underline{a};\underline{b};x,y;\alpha)
	&\coloneqq
	\sum_{\lambda\in\mathcal P_n} \frac{\poch{\underline{a}}{\lambda}}{\poch{\underline{b}}{\lambda}} \alpha^{|\lambda|} \Omega_\lambda(x;\alpha)J_\lambda^*(y;\alpha).\label{eqn:pFq-Mac-2-new}
\end{align}
Here $p,q\geq0$ are arbitrary, $\underline{a}=(a_1,\dots,a_p)$ and $\underline{b}=(b_1,\dots,b_q)$ are indeterminates.
We view these as \textit{formal power series}; analytic convergence has been studied in \cite[Proposition~1]{Kan93} and \cite[Theorem~6.5]{BR23}.
We often omit some parameters when there is no confusion.

By the stability relation \eqref{eqn:stable}, Jack hypergeometric series in one alphabet are also stable:
\begin{align}
	\hyper{p}{q}^{(m)}(\underline{a};\underline{b};x_1,\dots,x_m;\alpha)
	=\hyper{p}{q}^{(n)}(\underline{a};\underline{b};x_1,\dots,x_m,0,\dots,0;\alpha).
\end{align}
One may then define Jack hypergeometric series in infinitely many variables (in one alphabet), in the sense of \cite[Section~I.2]{Mac15}.
In this paper, we only work with the $n$-variate case and superscript $n$ will be omitted.

It follows directly that
\begin{align}
	\hyper{p}{q}(\underline{a};\underline{b};x,y;\alpha)=\hyper{p}{q}(\underline{a};\underline{b};y,x;\alpha),	\label{eqn:sym-xy}\quad
	\hyper{p}{q}(\underline{a};\underline{b};x,\bm1_n;\alpha)=	\hyper{p}{q}(\underline{a};\underline{b};x;\alpha).
\end{align}

Macdonald proved the following special cases in \cite[Section~6]{Mac-HG}:
\begin{align}
	\hyper{0}{0}(x;\alpha) &= \exp(p_1(x)),\label{eqn:0F0}	\\
	\hyper{1}{0}(a;x;\alpha) &= \prod_{i=1}^n (1-x_i)^{-a}.	\notag
\end{align}
See also \cite[Proposition~3.1]{Yan92}.
Macdonald noted that the Cauchy identity for Jack polynomials leads to the following summation formula:
\begin{align}
	\hyper{1}{0}(n/\alpha;x,y;\alpha) &= \sum_{\lambda\in\mathcal P_n} \frac{J_\lambda(x)J_\lambda(y)}{j_\lambda} = \prod_{i,j=1}^n(1-x_iy_j)^{-1/\alpha}. \notag
\end{align}

\subsection{Differential operators}
We recall some differential operators, all acting term-wise on formal power series.
Our notation is slightly different from Macdonald's \cite[p.~47]{Mac-HG}.

For $1\leq i\leq n$, let $\partial_i = \dpd{}{x_i}$.
Define
\begin{align}
	E_r &\coloneqq \sum_{i=1}^n x_i^{r-1} \partial_{i},	\quad r\geq1,	\\
	\square&\coloneqq  \frac{1}{2}\sum_{i=1}^n x_i^2\partial_{i}^2+\frac{1}{\alpha}\sum_{1\leq i\neq j\leq n} \frac{x_ix_j}{x_i-x_j}\partial_{i}.
\end{align}
Our $\square$ was written as $\square_2$ in \cite{Mac-HG}.
We shall often write $\sum_{i=1}^n$ as $\sum_{i}$ and $\sum_{1\leq i\neq j\leq n}$ as $\sum_{i\neq j}$.
Note that $E_2$ is the \mydef{Euler operator}, and $\square$ is the \mydef{Laplace--Beltrami operator} for Jack polynomials.
See \cite[Example~VI.3.3(e)]{Mac15}.
The Laplace--Beltrami operator is closely related to the Calogero--Moser--Sutherland model.
See \cite{Rosler03} for an overview.

For operators $A$ and $B$, define $\ad_{A}(B) \coloneqq [A,B] = AB-BA$, and recursively $\ad_{A}^r(B) \coloneqq \ad_A(\ad_A^{r-1}(B))$, for $r\geq1$, with $\ad_{A}^0(B)=B$.
We also recall $e_1(x)=p_1(x)=x_1+\dots+x_n$ acting by multiplication.

We collect some properties of these operators, most can be found in \cite{Las98,Mac-HG}.
We believe that \cref{eqn:square2-e1} is new, which is crucial in defining the raising operator $\R$ later.
This identity is inspired by \cref{eqn:square2-E1}, which was first given in \cite[p.~47, Remark]{Mac-HG}.
\begin{lemma}\label{lem}
	The following hold.
	\begin{enumerate}
		\item The operators $E_2$ and $\square$ act diagonally on $(J_\lambda)$ (or any form) by
		\begin{align}
			E_2(J_\lambda) &= |\lambda| \cdot J_\lambda,	\label{eqn:E2}\\
			\square(J_\lambda) &= \rho(\lambda)\cdot J_\lambda.\label{eqn:square2}
		\end{align}
		\item The operator $\exp (tE_1)$ acts by translation on $(\Omega_\lambda)$
		\begin{align}
			\exp(tE_1)(\Omega_\lambda(x;\alpha)) &=\Omega_\lambda(x+t\bm1_n;\alpha)	=\sum_{\mu\subseteq\lambda} t^{|\lambda|-|\mu|}\binom{\lambda}{\mu}\Omega_\mu(x;\alpha).
		\end{align}
		In particular,
		\begin{align}
			E_1(\Omega_\lambda) &=	\sum_{\mu\coveredby\lambda} \binom{\lambda}{\mu}\Omega_\mu.	\label{eqn:E1}
		\end{align}
		\item Multiplication by $\exp(te_1)$ acts on $(J_\mu^*)$ by
		\begin{align}
			\exp(te_1) \cdot J_\mu^* &=	\sum_{\lambda\supseteq\mu}(t\alpha)^{|\lambda|-|\mu|}\binom{\lambda}{\mu}J_\lambda^*.\label{eqn:exp(te1)}
		\end{align}
		In particular,
		\begin{align}
			e_1 \cdot J_\mu^* &=	\alpha\sum_{\lambda\cover\mu}\binom{\lambda}{\mu}J_\lambda^*.\label{eqn:e1}
		\end{align}
		\item
		For $r\geq0$, the commutators act by
		\begin{align}
			\(\ad_{-\square}^{r}(E_1)\)(\Omega_\lambda) &= \sum_{\mu\coveredby\lambda} \rho(\lambda/\mu)^{r} \binom{\lambda}{\mu}\Omega_\mu,\label{eqn:square2-E1}	\\
			\(\ad_{\square}^r(e_1)\)(J_\mu^*) &= \alpha \sum_{\lambda\cover\mu} \rho(\lambda/\mu)^{r} \binom{\lambda}{\mu}J_\lambda^*	\label{eqn:square2-e1}.
		\end{align}
	\end{enumerate}
\end{lemma}
\begin{proof}
	\begin{enumerate}
		\item See \cite[Example~VI.3.3(e)]{Mac15}.
		\item The univariate Taylor series is
		\begin{align*}
			f(\dots,x_i+t,\dots)
			&=f(\dots,x_i,\dots)+\partial_if(\dots,x_i,\dots)t+\partial_i^2f(\dots,x_i,\dots)\frac{t^2}{2!}+\cdots
			\\&=\exp(t\partial_{i})(f(x)),
		\end{align*}
		where ``$\dots$'' indicate that the variables are fixed.
		Since $\exp(tE_1) = \exp(t\partial_{1})\cdots \exp(t\partial_{n})$, the claim follows from iterating over $i=1,\dots,n$ and applying the binomial formula \eqref{eqn:bino}.
		\item \cref{eqn:exp(te1)} is Lassalle's definition \cite[p.~320]{Las98} of the binomial coefficients.
		We give a direct proof here.

		We first prove \cref{eqn:e1}.
		The Pieri rule for Macdonald polynomials is in \cite[VI.(6.7$'$)\&(6.13)]{Mac15}.
		Passing to the Jack case ($q=t^\alpha$, $t\to1$ with $\alpha$ fixed) and adjusting the normalizations gives
		\begin{align*}
			e_1 \cdot J_\mu^* = \sum_{\lambda\cover\mu} \(\prod_{i<i_0}\frac{\alpha(\mu_i-\mu_{i_0})+i_0-i-1}{\alpha(\mu_i-\mu_{i_0})+i_0-i}\frac{\alpha(\mu_i-\mu_{i_0}-1)+i_0-i+1}{\alpha(\mu_i-\mu_{i_0}-1)+i_0-i}\)\cdot \frac{c_\lambda'}{c_\mu'}J_\lambda^*,
		\end{align*}
		where $i_0=i_0(\lambda)$ is the row of $\lambda/\mu$.
		For each $\lambda\cover\mu$, let $s=(i,\mu_{i_0}+1)\in C$ be as in \cref{prop:bino}, then one checks
		\begin{align*}
			a_\mu(s)=a_\lambda(s)=\mu_i-\mu_{i_0}-1,\quad
			l_\mu(s)=l_\lambda(s)-1=i_0-i-1.
		\end{align*}
		Thus
		\begin{align*}
			\prod_{s\in C} \frac{c_\mu'(s)}{c_\lambda'(s)} &= \prod_{i<i_0} \frac{\alpha(\mu_i-\mu_{i_0})+i_0-i-1}{\alpha(\mu_i-\mu_{i_0})+i_0-i},	\\
			\prod_{s\in C} \frac{c_\lambda(s)}{c_\mu(s)} &= \prod_{i<i_0} \frac{\alpha(\mu_i-\mu_{i_0}-1)+i_0-i+1}{\alpha(\mu_i-\mu_{i_0}-1)+i_0-i}.
		\end{align*}
		By \cref{eqn:bino-comb},
		\begin{align*}
			\prod_{s\in C} \frac{c_\lambda(s)}{c_\mu(s)} \prod_{s\in C} \frac{c_\mu'(s)}{c_\lambda'(s)} \cdot \frac{c_\lambda'}{c_\mu'} = \alpha\binom{\lambda}{\mu},
		\end{align*}
		which proves \cref{eqn:e1}.

		Now we prove \cref{eqn:exp(te1)}.
		Define coefficients $b_{\lambda\mu}$ by
		\begin{align*}
			\exp(te_1) \cdot J_\mu^* &=	\sum_{\lambda}(t\alpha)^{|\lambda|-|\mu|}b_{\lambda\mu} J_\lambda^*.
		\end{align*}
		They satisfy the recursion
		\begin{align*}
			b_{\lambda\lambda}=1,\quad (|\lambda|-|\mu|)\cdot b_{\lambda\mu} = \sum_{\nu\coveredby\lambda} b_{\lambda\nu}b_{\nu\mu}.
		\end{align*}
		As shown in \cite{Sahi-Jack}, such coefficients are uniquely determined by the adjacent case $\lambda\cover\mu$.
		Since we have established that $b_{\lambda\mu}=\binom{\lambda}{\mu}$ for $\lambda\cover\mu$, they must also agree in general.
		\item
		We prove the first identity.
		The second is similar.
		Proceed by induction on $r$. The base case $r=0$ is \cref{eqn:E1}.
		For the inductive step,
		\begin{align*}
			\(\ad_{-\square}^{r+1}(E_1)\)(\Omega_\lambda)
			&=	[\ad_{-\square}^{r}(E_1),\square](\Omega_\lambda)
			\\&=	\ad_{-\square}^{r}(E_1)\(\square(\Omega_\lambda)\) -\square\(\ad_{-\square}^{r}(E_1)(\Omega_\lambda)\)
			\\&=	\ad_{-\square}^{r}(E_1)(\rho(\lambda)\Omega_\lambda) -\square\(\sum_{\mu\coveredby\lambda} \rho(\lambda/\mu)^{r} \binom{\lambda}{\mu}\Omega_\mu\)
			\\&=	\rho(\lambda)\sum_{\mu\coveredby\lambda} \rho(\lambda/\mu)^{r} \binom{\lambda}{\mu}\Omega_\mu -\sum_{\mu\coveredby\lambda} \rho(\lambda/\mu)^{r} \binom{\lambda}{\mu}\rho(\mu)\Omega_\mu
			\\&=	\sum_{\mu\coveredby\lambda} \rho(\lambda/\mu)^{r+1} \binom{\lambda}{\mu}\Omega_\mu.
		\end{align*}
	\end{enumerate}
\end{proof}

\begin{example}\label{ex:ad}
	For small $r$, the commutators admit explicit formulas:
	\begin{align*}
		\square_1&\coloneqq \ad_{-\square}(E_1) = [E_1,\square] =	\sum_{i} x_i\partial_{i}^2+\frac{1}{\alpha}\sum_{i\neq j} \frac{x_i+x_j}{x_i-x_j}\partial_{i},\\
		E_3 &= \ad_{\square}(e_1) = [\square,e_1] = \sum_i x_i^2\partial_i,	\\
		\ad_{-\square}^2(E_1) &=	\sum_{i} \(x_i^2\partial_i^3 +x_i\partial_{i}^2 +\frac{1}{\alpha}\sum_{j:j\neq i}\frac{x_i^2+2x_ix_j}{x_i-x_j}\partial_{i}^2 \)
		\\&\=	+\frac{1}{\alpha^2} \sum_i \sum_{j:j\neq i} \(\sum_{k:k\neq i,j}\frac{-(x_i+x_j)x_k^2+(x_i-x_j)^2x_k+x_ix_j(x_i+x_j)}{(x_i-x_j)(x_i-x_k)(x_j-x_k)} -\frac{x_i+x_j}{x_i-x_j}\)\partial_{i},	\\
		\ad_{\square}^2(e_1) &
		=	\sum_i \(x_i^3\partial_{i}^2+x_i^2\partial_{i}\) +\frac{2}{\alpha}\sum_{i\neq j}\frac{x_i^2x_j}{x_i-x_j}\partial_{i}.
	\end{align*}
	From \cref{lem} with $r=1$,
	\begin{align}
		\square_1(\Omega_\lambda) &=	[E_1,\square](\Omega_\lambda)	=\sum_{\mu\coveredby\lambda} \rho(\lambda/\mu)\binom{\lambda}{\mu}\Omega_\mu, \label{eqn:square1}	\\
		E_3(J_\mu^*) &=	[\square,e_1](J_\mu^*)	= \alpha \sum_{\lambda\cover\mu} \rho(\lambda/\mu)\binom{\lambda}{\mu}J_\lambda^*	\label{eqn:E3}.
	\end{align}
	\Cref{eqn:square1} appeared in \cite[p.~47, Lemma~1]{Mac-HG} and \cref{eqn:E3} in \cite[(14.1)]{Las98} (whose $E_2$ is our $E_3$).
\end{example}
\subsection{Debiard--Sekiguchi operators}
Now we recall the Debiard--Sekiguchi operators.
There are various different definitions, for example, \cite{OO97} and \cite[p.~319, Example~3]{Mac15}.
For us, it is most convenient to define the Debiard--Sekiguchi operator as
\begin{align}
	D(t) = D(t;\alpha) \coloneqq \frac{1}{V(x)}\det\(x_i^{n-j}\(x_i\partial_{i}-(j-1)/\alpha+t\)\)_{1\leq i,j\leq n},	\label{eqn:Sekiguchi}
\end{align}
which acts diagonally on $(J_\lambda)$ as
\begin{gather}
	D(t;\alpha) (J_\lambda) = \prod_{i=1}^n (\lambda_i-(i-1)/\alpha+t)\cdot J_\lambda.
\end{gather}
Define the operators $D_r$, for $0\leq r\leq n$, via the expansion
\begin{align}
	D(t) = \sum_{r=0}^n t^{n-r} D_r.
\end{align}
Then we have
\begin{align}
	D_r (J_\mu) = e_r(\tilde\mu) \cdot J_\mu,
\end{align}
where $\tilde\mu_i\coloneqq\mu_i-(i-1)/\alpha$ and $e_r$ is the $r$-th elementary symmetric polynomial.
In particular, the operators $D_0,D_1,\dots,D_n$ commute pairwise.
We set $D_r=0$ for $r>n$ in formulas involving the first few operators.
\begin{example}\label{ex:D}
	The first few operators are
	\begin{align*}
		D_0&=1,\\
		D_1&=E_2-\frac{n(n-1)}{2\alpha},\\
		D_2&=-\square+\frac12E_2(E_2-1)-\frac{n(n-1)}{2\alpha}E_2
		+\frac{n(n-1)(n-2)(3n-1)}{24\alpha^2}.
	\end{align*}
\end{example}
Note that Debiard--Sekiguchi operators $D_r$ depend on $n$.
For Debiard--Sekiguchi operators that work with infinitely many variables, see \cite{NS13}.

We now have all the ingredients to construct the differential operators $\L$, $\M$, $\N$ and $\R$.
Their constructions are in the next two sections.
\section{Two alphabets}\label{sec:2arg}
In this section we consider hypergeometric series with two alphabets $\hyper{p}{q}(\underline{a};\underline{b};x,y;\alpha)$.
We shall write $\square^{(x)}$ and $\square^{(y)}$ to indicate the differential operator $\square$ with respect to variables $x=(x_1,\dots,x_n)$ and $y=(y_1,\dots,y_n)$, respectively, and similarly for other operators.

We will sometimes omit the $\alpha$, $\underline{a}$ and $\underline{b}$ parameters in the expressions.
\subsection{Examples}
We first work out the examples of $\hyper{2}{1}$ in \cite{CM72} and \cite[p.~51]{Mac-HG} and $\hyper{3}{2}$ in \cite[(4.8)]{Fujikoshi}.
\subsubsection{$\hyper{2}{1}$}
Consider
\begin{align*}
	F(x,y)=\hyper{2}{1}(a,b;c;x,y;\alpha)=\sum_\lambda \frac{\poch{a}{\lambda}\poch{b}{\lambda}}{\poch{c}{\lambda}} \alpha^{|\lambda|}\Omega_\lambda(x)J_\lambda^*(y).
\end{align*}
We have
\begin{align*}
	\bigl([E_1^{(x)},\square^{(x)}]+cE_1^{(x)}\bigr)
	F
	&=\sum_\lambda\frac{\poch{a}{\lambda}\poch{b}{\lambda}}{\poch{c}{\lambda}}
	\alpha^{|\lambda|}\sum_{\mu\coveredby\lambda}
	\bigl(\rho(\lambda/\mu)+c\bigr)\binom{\lambda}{\mu}\Omega_\mu(x)J_\lambda^*(y)\\
	&=\sum_\mu\frac{\poch{a}{\mu}\poch{b}{\mu}}{\poch{c}{\mu}}\alpha^{|\mu|}
	\sum_{\lambda\cover\mu}\alpha
	\bigl(\rho(\lambda/\mu)+a\bigr)\bigl(\rho(\lambda/\mu)+b\bigr)
	\binom{\lambda}{\mu}\Omega_\mu(x)J_\lambda^*(y)\\
	&=\bigl([\square^{(y)},[\square^{(y)},e_1(y)]]
	+(a+b)[\square^{(y)},e_1(y)]+abe_1(y)\bigr)
	F.
\end{align*}
Let
\begin{align*}
	&\=	\hyper[\mathcal D]{2}{1}^{(x,y)} \notag
	\\&=	[E_1^{(x)},\square^{(x)}]+cE_1^{(x)}  -[\square^{(y)},[\square^{(y)},e_1(y)]]-(a+b)[\square^{(y)},e_1(y)]-abe_1(y)
	\\&=	\sum_i \(x_i\partial_{x_i}^2+c\partial_{x_i}-y_i^3\partial_{y_i}^2-(a+b+1)y_i^2\partial_{y_i}-aby_i\) +\frac1\alpha\sum_{i\neq j}\(\frac{x_i+x_j}{x_i-x_j}\partial_{x_i} -\frac{2y_i^2y_j}{y_i-y_j}\partial_{y_i}\)\notag
	\\&=	\sum_i \(\partial_{x_i}(x_i\partial_{x_i}+c-1)-y_i(y_i\partial_{y_i}+a)(y_i\partial_{y_i}+b)\) +\frac1\alpha\sum_{i\neq j}\(\frac{x_i+x_j}{x_i-x_j}\partial_{x_i} -\frac{2y_i^2y_j}{y_i-y_j}\partial_{y_i}\),\notag
\end{align*}
then
\begin{align*}
	\hyper[\mathcal D]{2}{1}^{(x,y)}(\hyper{2}{1}(x,y))=0.
\end{align*}
It is not hard to see that our operator $\hyper[\mathcal D]{2}{1}^{(x,y)}$ reduces to Constantine--Muirhead's operator \cite[(1.2)]{CM72} for the zonal case when setting $\alpha=2$.
Our operator $\hyper[\mathcal D]{2}{1}^{(x,y)}$ is slightly different from Macdonald's conjectural operator in \cite[p.~51]{Mac-HG} and is equal to the operator in \cite[Proposition~A.1]{BF97}.
\subsubsection{$\hyper{3}{2}$}
Now, let $\underline{a}=(a_1,a_2,a_3)$ and $\underline{b}=(b_1,b_2)$, and
\begin{align*}
	\hyper{3}{2}(\underline{a};\underline{b};x,y;\alpha)	=	\sum_\lambda \frac{\poch{\underline{a}}{\lambda}}{\poch{\underline{b}}{\lambda}} \alpha^{|\lambda|}\Omega_\lambda(x)J_\lambda^*(y).
\end{align*}
Write $\rho$ for $\rho(\lambda/\mu)$.
\begin{align*}
	&\=	\(\(\ad_{-\square^{(x)}}^2+(b_1+b_2)\ad_{-\square^{(x)}}+b_1b_2\)(E_1^{(x)})\)(\hyper{3}{2}(\underline{a};\underline{b};x,y;\alpha))
	\\&=	\sum_\lambda \frac{\poch{\underline{a}}{\lambda}}{\poch{\underline{b}}{\lambda}} \alpha^{|\lambda|}\sum_{\mu\coveredby\lambda}\(\rho^2+(b_1+b_2)\rho+b_1b_2\)\binom{\lambda}{\mu}\Omega_\mu(x)J_\lambda^*(y)
	\\&=	\sum_\mu \frac{\poch{\underline{a}}{\mu}}{\poch{\underline{b}}{\mu}} \alpha^{|\mu|} \sum_{\lambda\cover\mu}  \alpha\(\rho^3+e_1(\underline{a})\rho^2+e_2(\underline{a})\rho+e_3(\underline{a})\)\binom{\lambda}{\mu}\Omega_\mu(x)J_\lambda^*(y)
	\\&=	\(\(\ad_{\square^{(y)}}^3+e_1(\underline{a})\ad_{\square^{(y)}}^2+e_2(\underline{a})\ad_{\square^{(y)}}+e_3(\underline{a})\)(e_{1}(y))\)(\hyper{3}{2}(\underline{a};\underline{b};x,y;\alpha))
\end{align*}
Let
\begin{align*}
	\hyper[\mathcal D]{3}{2}^{(x,y)}
	&=	\(\ad_{-\square^{(x)}}^2+(b_1+b_2)\ad_{-\square^{(x)}}+b_1b_2\)(E_1^{(x)})	\notag
	\\&\=	-\(\ad_{\square^{(y)}}^3+e_1(\underline{a})\ad_{\square^{(y)}}^2+e_2(\underline{a})\ad_{\square^{(y)}}+e_3(\underline{a})\)(e_{1}(y)),
\end{align*}
then
\begin{align*}
	\hyper[\mathcal D]{3}{2}^{(x,y)}(\hyper{3}{2}(x,y))=0.
\end{align*}
By comparing the actions on Jack polynomials, our operator $\hyper[\mathcal D]{3}{2}^{(x,y)}$ reduces to Fujikoshi's operator, \cite[(4.8)]{Fujikoshi}, when $\alpha=2$.
While Fujikoshi’s operator is presented in a rather intricate and mysterious form, ours is quite straightforward to generalize.
\subsection{The general case}
Now we consider the general case:
\begin{align*}
	\hyper{p}{q}(\underline{a};\underline{b};x,y;\alpha) = \sum_\lambda \frac{\poch{\underline{a}}{\lambda}}{\poch{\underline{b}}{\lambda}} \alpha^{|\lambda|} \Omega_\lambda(x)J_\lambda^*(y).
\end{align*}
Define the following \mydef{lowering and raising differential operators}
\begin{align}
	\L^{(x)}=\hyper[\L]{}{q}^{(x)}(\underline{b})
	&\coloneqq	\sum_{r=0}^q e_{q-r}(\underline{b})\ad_{-\square^{(x)}}^{r}(E_1^{(x)})
	=\prod_{k=1}^q (\ad_{-\square^{(x)}}+b_k)(E_1^{(x)}),	\label{eqn:lower-x}	\\
	\R^{(y)}=\hyper[\R]{p}{}^{(y)}(\underline{a})
	&\coloneqq \sum_{r=0}^p e_{p-r}(\underline{a})\ad_{\square^{(y)}}^{r}(e_1(y))
	=\prod_{k=1}^p (\ad_{\square^{(y)}}+a_k)(e_1(y)).	\label{eqn:raise-y}
\end{align}

Let $\mathscr F^{(x,y)}$ be the space over $\mathbb K$ of formal power series in the form
\begin{align*}
	F(x,y)=\sum_{\eta,\zeta\in\Z_{\geq0}^n} C_{\eta,\zeta}(\alpha)x^\eta y^\zeta,
\end{align*}
where $C_{\eta,\zeta}(\alpha)\in\mathbb K$, and $F(x,y)$ is symmetric in $x$ and in $y$ separately.
We suppress the dependence of these coefficients on $\underline a$ and $\underline b$.
The differential operators $\hyper[\L]{}{q}^{(x)}(\underline{b})$ and $\hyper[\R]{p}{}^{(y)}(\underline{a})$, and the Debiard--Sekiguchi operators $D^{(x)}(t)$ and $D^{(y)}(t)$ are linear maps on the space $\mathscr F^{(x,y)}$.

Define $\mathscr F^{(x,y)}_{D}$ as the subspace of $\mathscr F^{(x,y)}$ that consists of formal power series in the form
\begin{align}\label{eqn:Fxy}
	F(x,y) = \sum_{\lambda} C_\lambda(\alpha) \alpha^{|\lambda|}  \Omega_\lambda(x)J_\lambda^*(y),
\end{align}
where $C_\lambda(\alpha)\in\mathbb K$.
\begin{prop}
	The following are equivalent for $F(x,y)\in\mathscr F^{(x,y)}$:
	\begin{enumerate}
		\item $F(x,y)\in\mathscr F^{(x,y)}_D$.
		\item $F(x,y)$ is in the kernel of $D^{(x)}(t)-D^{(y)}(t)$ for every $t$.
		\item $F(x,y)$ and $D^{(x)}(t)F(x,y)$ are $\tau$-invariant for every $t$, where $\tau:\mathscr F^{(x,y)}\to \mathscr F^{(x,y)}$ is the involution of interchanging $x$ and $y$.
	\end{enumerate}
\end{prop}
\begin{proof}
	(1)$\implies$(2) and (3) is obvious.
	We now prove that (2)$\implies$(1).
	Since $(J_\lambda(y))$ forms a basis, we may write $$F(x,y) = \sum_\lambda A_\lambda(x)J_\lambda(y)\in\mathscr F^{(x,y)}$$ for some symmetric formal power series $A_\lambda(x)$.
	Applying $D^{(x)}(t)-D^{(y)}(t)$, we get
	\begin{align*}
		\sum_\lambda \(D^{(x)}(t)(A_\lambda(x))-\prod_{i}(\lambda_i-(i-1)/\alpha+t)\cdot A_\lambda(x)\)J_\lambda(y)=0,
	\end{align*}
	which implies
	\begin{align*}
		D^{(x)}(t)(A_\lambda(x))=\prod_{i}(\lambda_i-(i-1)/\alpha+t)\cdot A_\lambda(x)
	\end{align*}
	for all $\lambda$.
	By the well-known fact that the Debiard--Sekiguchi operator characterizes Jack polynomials, we conclude that $A_\lambda(x)$ is a scalar multiple of $J_\lambda(x)$.

	As for (3)$\implies$(2), we note that $\tau D^{(x)}=D^{(y)}\tau$, and hence
	\begin{align*}
		D^{(x)}F -D^{(y)}F = D^{(x)}F -D^{(y)}\tau F = D^{(x)}F -\tau D^{(x)}F=0.&\qedhere
	\end{align*}
\end{proof}
\begin{maintheorem}\label{thm:A}
	The hypergeometric series $\hyper{p}{q}(\underline{a};\underline{b};x,y;\alpha)$ is the unique solution in $\mathscr F^{(x,y)}_D$ of the equation
	\begin{align}\label{eqn:LxRy(F)}
		\(\hyper[\L]{}{q}^{(x)}(\underline{b})-\hyper[\R]{p}{}^{(y)}(\underline{a})\) (F(x,y))=0,
	\end{align}
	subject to the initial condition that $F(\bm0_n,\bm0_n)=1$, i.e., $C_{(0)}(\alpha)=1$.
\end{maintheorem}
\begin{proof}
	We first prove uniqueness.
	The proof is based on \cite{CM72}.
	Let $F(x,y)$ be as in \cref{eqn:Fxy} and suppose that it satisfies \cref{eqn:LxRy(F)}.
	Then
	\begin{align*}
		\hyper[\L]{}{q}^{(x)}(\underline b)(F)
		&=\sum_\lambda C_\lambda(\alpha)\alpha^{|\lambda|}
		\sum_{\mu\coveredby\lambda}
		\frac{\poch{\underline b}{\lambda}}{\poch{\underline b}{\mu}}
		\binom{\lambda}{\mu}\Omega_\mu(x)J_\lambda^*(y),\\
		\hyper[\R]{p}{}^{(y)}(\underline a)(F)
		&=\sum_\lambda\alpha^{|\lambda|}
		\sum_{\mu\coveredby\lambda}C_\mu(\alpha)
		\frac{\poch{\underline a}{\lambda}}{\poch{\underline a}{\mu}}
		\binom{\lambda}{\mu}\Omega_\mu(x)J_\lambda^*(y).
	\end{align*}
	hence comparing the coefficients of $\Omega_\mu(x)J_\lambda^*(y)$ for each pair $\lambda\cover\mu$, we have
	\begin{align}\label{eqn:C-rec}
		C_\lambda(\alpha) \frac{\poch{\underline{b}}{\lambda}}{\poch{\underline{b}}{\mu}}
		&=	C_\mu(\alpha) \frac{\poch{\underline{a}}{\lambda}}{\poch{\underline{a}}{\mu}}.
	\end{align}
	The recursion \eqref{eqn:C-rec}, together with the initial condition, determines $(C_\lambda(\alpha))$ uniquely.

	By \cref{eqn:poch-ratio}, it is obvious that $C_\lambda(\alpha)=\frac{\poch{\underline{a}}{\lambda}}{\poch{\underline{b}}{\lambda}}$ satisfies the recursion \eqref{eqn:C-rec}, hence $\hyper{p}{q}(\underline{a};\underline{b};x,y;\alpha)$ is the unique solution.
\end{proof}

\subsection{Interchanging the variables}
By the symmetry under interchange of $x$ and $y$ in $\hyper{p}{q}(\underline{a};\underline{b};x,y;\alpha)$ (see \cref{eqn:sym-xy}), we may interchange the $x$ and $y$ variables in the operators as follows:
\begin{align}
	\hyper[\R]{p}{}^{(x)}(\underline{a})
	&\coloneqq \sum_{r=0}^p e_{p-r}(\underline{a})\ad_{\square^{(x)}}^{r}(e_1(x))
	= \prod_{k=1}^p (\ad_{\square^{(x)}}+a_k)(e_1(x))	,	\label{eqn:raise-x}\\
	\hyper[\L]{}{q}^{(y)}(\underline{b})
	&\coloneqq \sum_{r=0}^q e_{q-r}(\underline{b})\ad_{-\square^{(y)}}^{r}(E_1^{(y)})
	=\prod_{k=1}^q (\ad_{-\square^{(y)}}+b_k)(E_1^{(y)})	\label{eqn:lower-y}.
\end{align}
Following the reasoning in \cref{thm:A}, we derive
\begin{theoremp}\label{thm:A'}
	\Cref{thm:A} holds for the differential operator $\hyper[\L]{}{q}^{(y)}(\underline{b})-\hyper[\R]{p}{}^{(x)}(\underline{a})$.
\end{theoremp}
\section{One alphabet}\label{sec:1arg}
In this section, we consider hypergeometric series with one alphabet $\hyper{p}{q}(\underline{a};\underline{b};x;\alpha)$.
\subsection{Examples}\label{sec:1arg-1}
We begin with some lemmas and examples in \cite[p.~47--51]{Mac-HG}.
These are not needed for \cref{thm:B}, but will serve as useful illustrations.
\begin{lemma}\label{lem:square-p1}
	The operators $E_1,~E_2,~\square$, $[E_1,\square]$ and $[[E_1,\square],\square]$ act on $\exp(p_1)=\hyper{0}{0}(x;\alpha)$ by
	\begin{align*}
		E_1(\exp(p_1)) &= n\exp(p_1),	\\
		E_2(\exp(p_1)) &= p_1\exp(p_1),	\\
		\square(\exp(p_1)) &= \frac{1}{2}p_2\exp(p_1),
	\end{align*}
	and
	\begin{align*}
		[E_1,\square](\exp(p_1)) &= p_1\exp(p_1) =E_2(\exp(p_1)),	\\
		[[E_1,\square],\square](\exp(p_1)) &=\(\(1+\frac{n-1}{\alpha}\)p_1+p_2\)\exp(p_1)= \(\(1+\frac{n-1}{\alpha}\)E_2+2\square\)\exp(p_1).
	\end{align*}
\end{lemma}
\begin{proof}
	Direct computations.
\end{proof}
\begin{lemma}\label{lem:rho-bino}
	For fixed $\mu$, we have
	\begin{align*}
		\alpha\sum_{\lambda\cover\mu} \binom{\lambda}{\mu} \frac{J_\lambda^*(\bm1_n)}{J_\mu^*(\bm1_n)} &= n,	\\
		\alpha\sum_{\lambda\cover\mu} \rho(\lambda/\mu)\binom{\lambda}{\mu} \frac{J_\lambda^*(\bm1_n)}{J_\mu^*(\bm1_n)} &= |\mu|,	\\
		\alpha\sum_{\lambda\cover\mu} \rho(\lambda/\mu)^2\binom{\lambda}{\mu} \frac{J_\lambda^*(\bm1_n)}{J_\mu^*(\bm1_n)} &= \(1+(n-1)/\alpha\)|\mu|+2\rho(\mu).
	\end{align*}
\end{lemma}
\begin{proof}
	Apply the operators $E_1$, $[E_1,\square]$ and $[[E_1,\square],\square]$ to $\exp(p_1)=\hyper{0}{0}(x;\alpha)$ and compare the coefficients.
\end{proof}

Now consider
\begin{align*}
	F(x)=\hyper{2}{1}(a,b;c;x;\alpha)=\sum_\lambda \frac{\poch{a}{\lambda}\poch{b}{\lambda}}{\poch{c}{\lambda}} \alpha^{|\lambda|}J_\lambda^*(x).
\end{align*}
By \cref{lem}, we have
\begin{align*}
	\([E_1,\square]+cE_1\)\Omega_\lambda
	=	\sum_{\mu\coveredby\lambda} \(\rho(\lambda/\mu)+c\)\binom{\lambda}{\mu}\Omega_\mu,
\end{align*}
it follows that
\begin{align*}
	\bigl([E_1,\square]+cE_1\bigr)F
	&=\sum_\lambda\frac{\poch{a}{\lambda}\poch{b}{\lambda}}{\poch{c}{\lambda}}
	\alpha^{|\lambda|}J_\lambda^*(\bm1_n)
	\sum_{\mu\coveredby\lambda}
	\bigl(\rho(\lambda/\mu)+c\bigr)\binom{\lambda}{\mu}\Omega_\mu(x)\\
	&=\sum_\mu\frac{\poch{a}{\mu}\poch{b}{\mu}}{\poch{c}{\mu}}\alpha^{|\mu|}
	\biggl(\alpha\sum_{\lambda\cover\mu}
	\bigl(\rho(\lambda/\mu)+a\bigr)\bigl(\rho(\lambda/\mu)+b\bigr)
	\binom{\lambda}{\mu}\frac{J_\lambda^*(\bm1_n)}{J_\mu^*(\bm1_n)}\biggr)J_\mu^*(x).
\end{align*}
By \cref{lem:rho-bino},
\begin{align*}
	&\=\alpha\sum_{\lambda\cover\mu}
	\bigl(\rho(\lambda/\mu)+a\bigr)\bigl(\rho(\lambda/\mu)+b\bigr)
	\binom{\lambda}{\mu}\frac{J_\lambda^*(\bm1_n)}{J_\mu^*(\bm1_n)}\\
	&=2\rho(\mu)+\bigl(a+b+1+(n-1)/\alpha\bigr)|\mu|+abn
\end{align*}
and therefore
\begin{align*}
	\bigl([E_1,\square]+cE_1\bigr)F
	=\biggl(2\square+\biggl(a+b+1+\frac{n-1}{\alpha}\biggr)E_2+abn\biggr)
	F.
\end{align*}
Let
\begin{align*}
	\hyper[\mathcal D]{2}{1}^{(x)}
	&=	\([E_1,\square]+cE_1\) - \(2\square+\(a+b+1+\frac{n-1}{\alpha}\)E_2+abn\)
	\\&=	\sum_i \(x_i(1-x_i)\partial_{i}^2+\(c-(a+b+1)x_i\)\partial_{i}-ab\)
	+\frac{1}{\alpha}\sum_{i\neq j} \frac{(1-x_i)(x_i+x_j)}{x_i-x_j}\partial_{i}.
\end{align*}
Then
\begin{align*}
	\hyper[\mathcal D]{2}{1}^{(x)}(F)=0.
\end{align*}
This operator coincides with the operator $\hyper[\Phi]{2}{1}$ in \cite[p.~50]{Mac-HG}, where operators for $\hyper{1}{1}$ and $\hyper{0}{1}$ were also given.
Our $\hyper[\mathcal D]{2}{1}^{(x)}$ also coincides with the operator in \cite[Theorem~2.11]{Yan92} (whose $d$ is our $2/\alpha$).

\subsection{The general case: the lowering operator}\label{sec:1-arg-2}
The pattern in the example above indicates that to study the general case $\hyper{p}{q}$, we need to generalize \cref{lem:square-p1,lem:rho-bino}.

For any $r\geq0$, let
\begin{align}
	f_r(\mu) &\coloneqq \alpha\sum_{\lambda\cover\mu} \rho(\lambda/\mu)^r \binom{\lambda}{\mu}\frac{j_\mu}{j_\lambda},	\label{eqn:fr-1}\\
	g_{r,n}(\mu) &\coloneqq \alpha\sum_{\lambda\cover\mu}\rho(\lambda/\mu)^r \binom{\lambda}{\mu}\frac{J_\lambda^*(\bm1_n)}{J_\mu^*(\bm1_n)}. \label{eqn:grn-1}
\end{align}
In \cref{lem:square-p1,lem:rho-bino}, Macdonald constructed eigen-operators $\mathcal G_{r,n}$ acting on $(J_\mu)$ with prescribed eigenvalues $g_{r,n}(\mu)$ for $r\leq2$.
His construction relied on iterated commutators of $E_1$ and $\square$ acting on $\hyper{0}{0}$, together with combinations of the first two Debiard--Sekiguchi operators $E_2$ and $\square$.
To go further, we experimented with higher commutators and found that the higher Debiard--Sekiguchi operators are also required.
For $r=3$ this approach still succeeds, although the computations are quite involved and the outcome does not suggest a general pattern.

We therefore turn to a generating-function approach, which will allow us to construct all $\mathcal G_{r,n}$ for $r\ge0$ in a unified way.

\subsubsection{The function $g_{r,n}(\mu)$}
We first investigate the eigenvalue $g_{r,n}(\mu)$.
Since
\begin{align*}
	\frac{J_\lambda(\bm1_n)}{J_\mu(\bm1_n)}
	= \frac{\alpha^{|\lambda|}\poch{n/\alpha}{\lambda}}{\alpha^{|\mu|}\poch{n/\alpha}{\mu}}
	= \alpha\cdot (n/\alpha+\rho(\lambda/\mu))
	= n+\alpha\rho(\lambda/\mu),
\end{align*}
we have
\begin{align}\label{eqn:gf-rec}
	g_{r,n} = \alpha f_{r+1}+nf_{r},
\end{align}

Define $\phi_{\lambda/\mu}$ as the Pieri coefficient for the integral form:
\begin{align}
	e_1 \cdot J_\mu = \sum_\lambda \phi_{\lambda/\mu} J_\lambda,
\end{align}
then by \cref{eqn:e1}, we have
\begin{align}
	\phi_{\lambda/\mu} =\alpha\binom{\lambda}{\mu}\frac{j_\mu}{j_\lambda},\quad \lambda\cover\mu,
\end{align}
and thus
\begin{align}
	f_r(\mu) = \sum_{\lambda\cover\mu} \phi_{\lambda/\mu} \rho(\lambda/\mu)^r.
\end{align}
In this subsection only, let
\begin{align}
	w_i=w_i(\mu)\coloneqq\mu_i-(i-1)/\alpha,
\end{align}
then $w_i=\rho((\mu+\varepsilon_{i})/\mu)$, if $\mu+\varepsilon_i$ is a partition, where $\varepsilon_i$ is the $i$-th unit vector $(0,\dots,0,1,0,\dots,0)$.
\begin{lemma}\label{lem:phi}
	Let $\lambda=\mu+\varepsilon_{i_0}$ be a partition for some $i_0=1,\dots,n$, then
	\begin{align}
		\phi_{\lambda/\mu}
		&= \frac{1/\alpha}{w_{i_0}+n/\alpha} \prod_{\substack{i=1\\i\neq i_0}}^n \frac{w_{i_0}-w_i+1/\alpha}{w_{i_0}-w_i}
		=	\frac{1}{\mu_{i_0}\alpha+n-i_0+1} \prod_{\substack{i=1\\i\neq i_0}}^n \frac{(\mu_{i_0}-\mu_i)\alpha+i-i_0+1}{(\mu_{i_0}-\mu_i)\alpha+i-i_0}.
	\end{align}
\end{lemma}
\begin{proof}
	Let $R$ and $C$ be as in \cref{prop:bino}, then we have (noting that $c_\lambda(i_0,\mu_{i_0}+1)=1$ and $c_\lambda'(i_0,\mu_{i_0}+1)=\alpha$)
	\begin{align*}
		\phi_{\lambda/\mu} =
		\prod_{s\in R}\frac{c_\mu(s)}{c_\lambda(s)} \cdot \prod_{s\in C} \frac{c_\mu'(s)}{c_\lambda'(s)}.
	\end{align*}
	Since
	\begin{align}
		\prod_{s\in \textup{row }i} c_\lambda(s) = \alpha^{\lambda_i}\cdot \poch{(n-i+1)/\alpha}{\lambda_i} \cdot \prod_{j=i+1}^n \frac{\poch{(j-i)/\alpha}{\lambda_i-\lambda_j}}{\poch{(j-i+1)/\alpha}{\lambda_i-\lambda_j}},
	\end{align}
	we have
	\begin{align*}
		\prod_{s\in R}\frac{c_\mu(s)}{c_\lambda(s)}
		&=	\frac{\prod_{s\in\text{row $i_0$ of $\mu$}}c_\mu(s)}{\prod_{s\in\text{row $i_0$ of $\lambda$}}c_\lambda(s)}
		=	\frac{1/\alpha}{(n-{i_0}+1)/\alpha+\mu_{i_0}}\prod_{j=i_0+1}^n \frac{(j-{i_0}+1)/\alpha+\mu_{i_0}-\mu_j}{(j-{i_0})/\alpha+\mu_{i_0}-\mu_j}
		\\&=	\frac{1/\alpha}{w_{i_0}+n/\alpha} \prod_{i=i_0+1}^n \frac{w_{i_0}-w_i+1/\alpha}{w_{i_0}-w_i}.
	\end{align*}
	For the column $C$, we have
	\begin{align*}
		\prod_{s\in C} \frac{c_\mu'(s)}{c_\lambda'(s)}
		=	\prod_{i=1}^{i_0-1} \frac{(\mu_{i}-\mu_{i_0})\alpha+{i_0}-i-1}{(\mu_{i}-\mu_{i_0})\alpha+i_0-i}
		=	\prod_{i=1}^{i_0-1} \frac{w_{i_0}-w_i+1/\alpha}{w_{i_0}-w_i}.
	\end{align*}
	The desired identity follows immediately.
\end{proof}

Note that if $\mu+\varepsilon_{i_0}$ is not a partition, that is, if $i_0>1$ and $\mu_{i_0-1}=\mu_{i_0}$, then the expression in \cref{lem:phi} is zero, since the numerator indexed by $i=i_0-1$ is 0.

Define the generating functions of $(f_r)$ and $(g_{r,n})$ as
\begin{align}
	F(\mu;s)	\coloneqq \sum_{r=0}^\infty f_r(\mu)s^r
	\text{\quad and \quad}
	G_n(\mu;s)	\coloneqq \alpha+\sum_{r=0}^\infty g_{r,n}(\mu) s^{r+1}.
	\label{eqn:GF-1}
\end{align}
\begin{prop}\label{lem:F-GF}
	The generating function $F(\mu;s)$ is
	\begin{align}\label{eqn:F-2}
		F(\mu;s)
		&=\frac{1}{1+ns/\alpha} \(\prod_{i=1}^n \frac{1-(w_i-1/\alpha)s}{1-w_is}-\prod_{i=1}^n\frac{w_i+(n-1)/\alpha}{w_i+n/\alpha}\)
		\\&= \frac{1}{1+ns/\alpha} \(\prod_{i=1}^n \frac{1-(\mu_i-i/\alpha)s}{1-(\mu_i-(i-1)/\alpha)s}-\prod_{i=1}^n\frac{\mu_i+(n-i)/\alpha}{\mu_i+(n-i+1)/\alpha}\).\notag
	\end{align}
	In particular, $f_0(\mu)=F(\mu;s=0)=1-\prod_{i=1}^n\frac{w_i+(n-1)/\alpha}{w_i+n/\alpha}$.
	Note that if $n>\ell(\mu)$, then $w_n=-(n-1)/\alpha$ and the subtrahend $\prod_{i=1}^n\frac{w_i+(n-1)/\alpha}{w_i+n/\alpha}$ vanishes.
\end{prop}

\begin{proof}
	We have
	\begin{align*}
		F(\mu;s)	&\coloneqq \sum_{r=0}^\infty f_r(\mu)s^r
		=	\sum_{\lambda\cover\mu}\sum_{r=0}^\infty \phi_{\lambda/\mu}\rho(\lambda/\mu)^r s^r
		=	\sum_{\lambda\cover\mu} \frac{\phi_{\lambda/\mu}}{1-\rho(\lambda/\mu)s}
		\\&=	\sum_{i_0=1}^n  \frac{1}{1-w_{i_0}s}\frac{1/\alpha}{w_{i_0}+n/\alpha} \prod_{\substack{i=1\\i\neq i_0}}^n \frac{w_{i_0}-w_i+1/\alpha}{w_{i_0}-w_i}.
	\end{align*}
	We first regard the $w_i$ as independent variables, so that the following partial fraction calculation is valid away from its poles.
	The resulting rational identity then specializes to the required $w_i(\mu)$.
	Write
	\begin{align*}
		\frac{1}{1+ns/\alpha} \prod_{i=1}^n \frac{1-(w_i-1/\alpha)s}{1-w_is}
		= \frac{B_0}{1+ns/\alpha}+\sum_{i=1}^n \frac{B_i}{1-w_is},
	\end{align*}
	with undetermined coefficients $B_i$.
	For each $i_0=1,\dots,n$, multiply by $1-w_{i_0}s$ and evaluate at $s=1/w_{i_0}$ (note that $w_i$'s are pairwise distinct), we have
	\begin{align*}
		B_{i_0}	=	\frac{1/\alpha}{w_{i_0}+n/\alpha} \prod_{i\neq i_0} \frac{w_{i_0}-w_i+1/\alpha}{w_{i_0}-w_i},\quad i_0=1,\dots, n.
	\end{align*}
	Similarly, we have $B_0=\prod_{i=1}^n\frac{w_i+(n-1)/\alpha}{w_i+n/\alpha}$, giving the subtrahend in \cref{eqn:F-2}.
\end{proof}
\begin{prop}
	The generating function $G_n(\mu;s)$ is
	\begin{align}\label{eqn:Gn-2}
		G_n(\mu;s) &= \alpha\prod_{i=1}^n \frac{1-(w_i-1/\alpha)s}{1-w_is}
		=\alpha\prod_{i=1}^n \frac{1-(\mu_i-i/\alpha)s}{1-(\mu_i-(i-1)/\alpha)s},
	\end{align}
	then for $r\geq0$,
	\begin{align}\label{eqn:grn-2}
		g_{r,n}(\mu) =  \alpha\sum_{p=0}^{r+1} (-1)^{p}e_{p}(w_1-1/\alpha,\dots,w_n-1/\alpha)h_{r+1-p}(w_1,\dots,w_n),
	\end{align}
	where $h_k$ is the complete homogeneous symmetric polynomial of degree $k$.
\end{prop}
\begin{proof}
	By \cref{eqn:gf-rec,eqn:GF-1}, we have
	\begin{align*}
		G_n(\mu;s) =	\alpha\((1+ns/\alpha)F(\mu;s)+1-f_0(\mu)\),
	\end{align*}
	then \cref{eqn:Gn-2} follows from \cref{lem:F-GF}.
	To find $g_{r,n}$, we could expand
	\begin{align*}
		G_n(\mu;s)=	\alpha\sum_{r=0}^{n} e_r(w_1-1/\alpha,\dots,w_n-1/\alpha)(-s)^r \sum_{r=0}^\infty h_r(w_1,\dots,w_n)s^r,
	\end{align*}
	then comparing the coefficient of $s^{r+1}$ gives \cref{eqn:grn-2}.
\end{proof}
\subsubsection{The operators $\mathcal G_{r,n}$ and $\hyper[\M]{p}{}(\underline{a})$}
Recall that the Debiard--Sekiguchi operator $D(t)$ acts on $(J_\mu)$ by
\begin{align*}
	D(t)(J_\mu) = \prod_{i=1}^n(\mu_i-(i-1)/\alpha+t) J_\mu =\prod_{i=1}^n(w_i+t) \cdot J_\mu,
\end{align*}
and is expanded as
\begin{align*}
	D(t) = \sum_{r=0}^n t^{n-r} D_r.
\end{align*}

Now, we have
\begin{align*}
	\alpha\frac{D(-1/s-1/\alpha)}{D(-1/s)}(J_\mu) =
	\alpha\prod_{i=1}^n \frac{1-(w_i-1/\alpha)s}{1-w_is} \cdot J_\mu
	=G_n(\mu;s) \cdot J_\mu.
\end{align*}
Define the operator $\mathcal G_n(s) $ and expand it as a Taylor series (in $s$ near $s=0$) as
\begin{align}\label{eqn:Gn}
	\mathcal G_n(s)
	&\coloneqq \alpha\frac{D(-1/s-1/\alpha)}{D(-1/s)}
	=	\alpha \frac{\sum_{r=0}^n (-s)^{r}(1+s/\alpha)^{n-r}D_r}{\sum_{r=0}^n (-s)^rD_r}	\notag
	\\&=	\alpha\sum_{r=0}^n  (-s)^{r}(1+s/\alpha)^{n-r}D_r \cdot \sum_{k=0}^\infty \(-\sum_{r=1}^n (-s)^rD_r\)^k.
\end{align}
Define $\mathcal G_{r,n}$ by
\begin{align}
	\mathcal G_n(s) \eqqcolon \alpha+\sum_{r=0}^\infty \mathcal G_{r,n}s^{r+1}.
\end{align}
Then by the very definition, we have
\begin{align}
	\mathcal G_{n}(s)(J_\mu) = G_n(\mu;s)\cdot J_\mu,	\text{\quad and \quad}
	\mathcal G_{r,n}(J_\mu) = g_{r,n}(\mu)\cdot J_\mu.
\end{align}
\begin{example}\label{ex:G}
	The first few operators $\mathcal G_{r,n}$ are:
	\begin{align*}
		\mathcal G_{0,n} &=	n,	\\
		\mathcal G_{1,n} &=	D_1+\frac{n(n-1)}{2\alpha} = E_2,	\\
		\mathcal G_{2,n} &=	D_1^2-2D_2+\frac{n-1}{\alpha} D_1+\frac{n(n-1)(n-2)}{6\alpha^2} =  2\square +\(1+(n-1)/\alpha\)E_2,	\\
		\mathcal G_{3,n} &=	\(D_{1}^{3}-3D_{1}D_{2}+3D_{3}\) +\frac{1}{\alpha}\((n-1)D_{1}^{2}+(3-2n)D_{2}\) \notag
		\\&\=	+\frac{(n-1)(n-2)D_{1}}{2\alpha^2} +\frac{n(n-1)(n-2)(n-3)}{24\alpha^3}.
	\end{align*}
	Note that $\mathcal G_{r,n}$ for $r=0,1,2$ coincide with the operators Macdonald constructed.
\end{example}
Define the operator
\begin{align}\label{eqn:pD1}
	\hyper[\M]{p}{} = \hyper[\M]{p}{}(\underline{a})
	\coloneqq	\sum_{r=0}^p e_{p-r}(\underline{a})\mathcal G_{r,n}
	=\CT\(s^{-1}\prod_{k=1}^p (a_k+s^{-1}) \cdot \mathcal G_{n}(s)\),
\end{align}
where $\CT$ denotes the constant term of the Laurent series in $s$ near $s=0$.
Then $\hyper[\M]{p}{}(\underline{a}) $ acts diagonally on $(J_\mu)$ by
\begin{align}
	\hyper[\M]{p}{}(\underline{a})(J_\mu)
	=	\sum_{r=0}^p e_{p-r}(\underline{a})g_{r,n}(\mu) \cdot J_\mu
	=	\alpha\sum_{\lambda\cover\mu} \frac{\poch{\underline{a}}{\lambda}}{\poch{\underline{a}}{\mu}} \binom{\lambda}{\mu}\frac{J_\lambda^*(\bm1_n)}{J_\mu^*(\bm1_n)}\cdot J_\mu.
\end{align}
\subsubsection{\texorpdfstring{\cref{thm:B}}{Theorem B}}
For this subsection, we shall vary the number of variables.
The operators $\hyper[\L]{}{q}(\underline{b})$ given by \cref{eqn:lower-x} and $\hyper[\M]{p}{}(\underline{a})$ by \cref{eqn:pD1} are in $n$ variables, which will be denoted by $\hyper[\L]{}{q}^{(n)}(\underline{b})$ and $\hyper[\M]{p}{}^{(n)}(\underline{a})$, respectively.
For $1\leq m\leq n$, let $\hyper[\L]{}{q}^{(m)}(\underline{b})$ and $\hyper[\M]{p}{}^{(m)}(\underline{a})$ be the corresponding operators in the variables $(x_1,\dots,x_m)$.

For any $m\leq n$, let $\mathscr F^{(x_1,\dots,x_m)}$ be the space over $\mathbb K$ of formal power series in the form
\begin{align}\label{eqn:F(x)}
	F(x_1,\dots,x_m) = \sum_{\lambda\in\mathcal P_m} C_\lambda(\alpha) \alpha^{|\lambda|} J_\lambda^*(x_1,\dots,x_m;\alpha),
\end{align}
where each $C_\lambda(\alpha)$ is in $\mathbb K$.
The differential operators $\hyper[\M]{p}{}^{(m)}(\underline{a})$ and $\hyper[\L]{}{q}^{(m)}(\underline{b})$ are linear maps on the spaces $\mathscr F^{(x_1,\dots,x_m)}$ for each $m$.
Thanks to \cref{eqn:stable}, if $F(x_1,\dots,x_n)\in \mathscr F^{(x_1,\dots,x_n)}$, then $$F(x_1,\dots,x_m,0,\dots,0)\in\mathscr F^{(x_1,\dots,x_m)}.$$
Write $\mathscr F^{(x)}=\mathscr F^{(x_1,\dots,x_n)}$.

\begin{maintheorem}\label{thm:B}
	The hypergeometric series $\hyper{p}{q}(\underline{a};\underline{b};x_1,\dots,x_n;\alpha)$ is the unique solution in $\mathscr F^{(x)}$ of the equation
	\begin{align}\label{eqn:pDq1x(F)}
		\(\hyper[\L]{}{q}^{(n)}(\underline{b})-\hyper[\M]{p}{}^{(n)}(\underline{a})\)(F(x_1,\dots,x_n))=0,
	\end{align}
	subject to
	\begin{enumerate}
		\item the initial condition that $F(\bm0_n)=1$, i.e., $C_{(0)}(\alpha)=1$.
		\item the stability condition that, for every $1\leq m<n$, the specialization $F_m$ defined by
		\begin{align*}
			F_m(x_1,\dots,x_m)\coloneqq F(x_1,\dots,x_m,0,\dots,0)
		\end{align*} is a solution of the equation
		\begin{align}\label{eqn:pDq1x(F)-m}
			\(\hyper[\L]{}{q}^{(m)}(\underline{b})-\hyper[\M]{p}{}^{(m)}(\underline{a})\)(F_m(x_1,\dots,x_m))=0.
		\end{align}
		(When $m$ is $n$, this is the same as \cref{eqn:pDq1x(F)}.)
	\end{enumerate}
\end{maintheorem}
We remark that the stability condition is inspired by \cite[Theorem~4.10]{Kan96} in the Macdonald polynomial setting.
See \cref{sec:future}.
\begin{proof}
	Let $F$ satisfy the equation and the stability condition in the theorem.
	For $1\leq m\leq n$, write $F(x_1,\dots,x_m)=F_m(x_1,\dots,x_m)$ in this proof.
	By \cref{eqn:stable}, the coefficients $C_\lambda(\alpha)$ in \cref{eqn:F(x)} are independent of $m$.
	Then
	\begin{align*}
		\hyper[\L]{}{q}^{(m)}(\underline b)(F_m)
		&=\sum_{\lambda\in\mathcal P_m}C_\lambda(\alpha)\alpha^{|\lambda|}
		\sum_{\mu\coveredby\lambda}
		\frac{\poch{\underline b}{\lambda}}{\poch{\underline b}{\mu}}
		\binom{\lambda}{\mu}
		\frac{J_\lambda^*(\bm1_m)}{J_\mu^*(\bm1_m)}
		J_\mu^*(x_1,\dots,x_m)\\
		&=\sum_{\mu\in\mathcal P_m}\alpha^{|\mu|+1}
		\biggl(\sum_{\lambda\cover\mu}C_\lambda(\alpha)
		\frac{\poch{\underline b}{\lambda}}{\poch{\underline b}{\mu}}
		\binom{\lambda}{\mu}
		\frac{J_\lambda^*(\bm1_m)}{J_\mu^*(\bm1_m)}\biggr)
		J_\mu^*(x_1,\dots,x_m),\\
		\hyper[\M]{p}{}^{(m)}(\underline a)(F_m)
		&=\sum_{\mu\in\mathcal P_m}C_\mu(\alpha)\alpha^{|\mu|+1}
		\biggl(\sum_{\lambda\cover\mu}
		\frac{\poch{\underline a}{\lambda}}{\poch{\underline a}{\mu}}
		\binom{\lambda}{\mu}
		\frac{J_\lambda^*(\bm1_m)}{J_\mu^*(\bm1_m)}\biggr)
		J_\mu^*(x_1,\dots,x_m).
	\end{align*}
	Hence, comparing the coefficients of $J_\mu^*(x_1,\dots,x_m)$ for each $\mu\in\mathcal P_m$, we have
	\begin{equation}\label{eqn:rec-up}
		\sum_{\lambda\cover\mu}C_\lambda(\alpha)
		\frac{\poch{\underline b}{\lambda}}{\poch{\underline b}{\mu}}
		\binom{\lambda}{\mu}\frac{J_\lambda^*(\bm1_m)}{J_\mu^*(\bm1_m)}
		=C_\mu(\alpha)\sum_{\lambda\cover\mu}
		\frac{\poch{\underline a}{\lambda}}{\poch{\underline a}{\mu}}
		\binom{\lambda}{\mu}\frac{J_\lambda^*(\bm1_m)}{J_\mu^*(\bm1_m)}.
	\end{equation}
	By definitions of $\alpha$-Pochhammer symbol and \cref{eqn:poch-ratio}, we see that $C_\lambda(\alpha)=\frac{\poch{\underline{a}}{\lambda}}{\poch{\underline{b}}{\lambda}}$ satisfies \cref{eqn:rec-up} and the initial condition, and hence $\hyper{p}{q}(\underline{a};\underline{b};x_1,\dots,x_m,0,\dots,0;\alpha)$ is a solution of the equation \eqref{eqn:pDq1x(F)-m}.

	In what follows, we prove the uniqueness.
	Let
	\begin{align*}
		K_{\underline{a}}(\lambda,\mu) \coloneqq \frac{\poch{\underline{a}}{\lambda}}{\poch{\underline{a}}{\mu}}\binom{\lambda}{\mu} \frac{j_\mu}{j_\lambda},
		\quad
		K_{\underline{b}}(\lambda,\mu) \coloneqq \frac{\poch{\underline{b}}{\lambda}}{\poch{\underline{b}}{\mu}}\binom{\lambda}{\mu} \frac{j_\mu}{j_\lambda}.
	\end{align*}
	For $\mu\neq(0)$ with $\ell(\mu)<n$, use \cref{eqn:rec-up} at $m=\ell(\mu)$ and $m=\ell(\mu)+1$.
	The ratio
	\begin{align*}
		\frac{J_\lambda^*(\bm1_m)}{J_\mu^*(\bm1_m)}
		=\frac{j_\mu}{j_\lambda}\bigl(m+\alpha\rho(\lambda/\mu)\bigr)
	\end{align*}
	is linear in $m$.
	We may extend both sums to all $\lambda\in\mathcal P_n$ covering $\mu$, because a term with $\ell(\lambda)=m+1$ has $m+\alpha\rho(\lambda/\mu)=0$.
	Subtracting the two equations and then comparing their constant terms gives
	\begin{align}
		\sum_{\lambda\cover\mu} C_\lambda(\alpha)K_{\underline b}(\lambda,\mu)
		&=C_\mu(\alpha)\sum_{\lambda\cover\mu}K_{\underline a}(\lambda,\mu),
		\tag{I}\label{eqn:I}\\
		\sum_{\lambda\cover\mu} C_\lambda(\alpha)\rho(\lambda/\mu)K_{\underline b}(\lambda,\mu)
		&=C_\mu(\alpha)\sum_{\lambda\cover\mu}\rho(\lambda/\mu)K_{\underline a}(\lambda,\mu).
		\tag{II}\label{eqn:II}
	\end{align}
	When $\ell(\mu)=n$, only the equation with $m=n$ is available, and we retain \cref{eqn:rec-up} itself.

	For $\mu=(0)$, use \cref{eqn:rec-up} with $m=1$ and $C_{(0)}(\alpha)=1$ to obtain
	\begin{align*}
		C_{(1)}(\alpha)=\frac{\prod_{k=1}^p a_k}{\prod_{k=1}^q b_k}
		=\frac{\poch{\underline a}{(1)}}{\poch{\underline b}{(1)}}.
	\end{align*}
	Now suppose that all $C_\mu(\alpha)$ with $|\mu|\leq d$ are known, where $d\geq1$.
	List the partitions of size $d$ and length at most $n$ in decreasing lexicographic order as $\mu^1>_L\cdots>_L\mu^K$.
	Here $\nu>_L\mu$ means that the first nonzero entry of $(\nu_i-\mu_i)_i$ is positive.
	Let $S(\mu)=\{\lambda\in\mathcal P_n:\lambda\cover\mu\}$.

	For each $\mu=\mu^k$, put $l=\ell(\mu)$.
	If $i<l$ and $\lambda=\mu+\varepsilon_i$ is a partition, then $\nu=\mu+\varepsilon_i-\varepsilon_l$ is a partition of size $d$ preceding $\mu$, and $\lambda\in S(\nu)$.
	Consequently, after processing $\mu^1,\dots,\mu^{k-1}$, the only possible new coefficients in $S(\mu^k)$ correspond to $\mu^k+\varepsilon_l$ and, when $l<n$, $\mu^k+\varepsilon_{l+1}$.

	If $l<n$, \cref{eqn:I,eqn:II} determine these at most two new coefficients.
	For two new partitions $\lambda^1,\lambda^2$, the determinant is
	\begin{align*}
		K_{\underline b}(\lambda^1,\mu)K_{\underline b}(\lambda^2,\mu)
		\bigl(\rho(\lambda^2/\mu)-\rho(\lambda^1/\mu)\bigr),
	\end{align*}
	which is nonzero in $\mathbb K$.
	Indeed, the adjacent binomial coefficients and the factors involving the indeterminates $b_k$ are nonzero, and distinct added boxes have distinct $\alpha$-contents.
	If only one coefficient is new, \cref{eqn:I} suffices.

	If $l=n$, there is at most one new coefficient.
	\Cref{eqn:rec-up} at $m=n$ determines it, since the coefficient for an added box in row $i$ contains the nonzero factor
	\begin{align*}
		n+\alpha\rho(\lambda/\mu)=n-i+1+\alpha\mu_i.
	\end{align*}
	Thus induction on $k$ determines every coefficient of size $d+1$.
	Every such partition covers a partition of size $d$, so induction on $d$ proves uniqueness.
\end{proof}
\begin{example}
	Assume $n\geq2$.
	In the proof, we examined $\mu=(0)$.
	Now, consider $\mu=(1)$, then $\lambda^1=(2)$ and $\lambda^2=(1,1)$ are the partitions covering $\mu$.
	Note that $\rho(\lambda^1/\mu)=1$, $\rho(\lambda^2/\mu)=-1/\alpha$, $\binom{\lambda^1}{\mu}=\binom{\lambda^2}{\mu}=2$, $\frac{j_\mu}{j_{\lambda^1}}=\frac{1}{2\alpha(\alpha+1)}$, and $\frac{j_\mu}{j_{\lambda^2}}=\frac{1}{2(\alpha+1)}$, then \cref{eqn:rec-up} becomes
	\begin{align*}
		&\=	C_{\lambda^1}(\alpha) \prod_{k=1}^q \(1+b_k\) \frac{m+\alpha}{\alpha(\alpha+1)}
		+C_{\lambda^2}(\alpha) \prod_{k=1}^q \(-1/\alpha+b_k\) \frac{m-1}{\alpha+1}
		\\&=C_{\mu}(\alpha)\cdot \(\prod_{k=1}^p \(1+a_k\) \frac{m+\alpha}{\alpha(\alpha+1)}+\prod_{k=1}^p \(-1/\alpha+a_k\) \frac{m-1}{\alpha+1}\),
	\end{align*}
	comparing the coefficients of $m$ and the constant terms gives \cref{eqn:I,eqn:II} for $\mu=(1)$:
	\begin{align*}
		C_{\lambda^1}(\alpha) \prod_{k=1}^q \(1+b_k\) -C_{\lambda^2}(\alpha) \prod_{k=1}^q \(-1/\alpha+b_k\)
		&=	C_{\mu}(\alpha)\cdot \(\prod_{k=1}^p \(1+a_k\) -\prod_{k=1}^p \(-1/\alpha+a_k\) \) ,	\\
		C_{\lambda^1}(\alpha) \prod_{k=1}^q \(1+b_k\) +\alpha C_{\lambda^2}(\alpha) \prod_{k=1}^q \(-1/\alpha+b_k\)
		&=	C_{\mu}(\alpha)\cdot \(\prod_{k=1}^p \(1+a_k\)+\alpha\prod_{k=1}^p \(-1/\alpha+a_k\) \),
	\end{align*}
	hence
	\begin{align*}
		C_{\lambda^1}(\alpha) &= C_{\mu}(\alpha)\cdot \frac{\prod_{k=1}^p \(1+a_k\)}{\prod_{k=1}^q \(1+b_k\)}
		=\frac{\poch{\underline{a}}{\lambda^1}}{\poch{\underline{b}}{\lambda^1}},	\\
		C_{\lambda^2}(\alpha) &= C_{\mu}(\alpha)\cdot \frac{\prod_{k=1}^p \(-1/\alpha+a_k\)}{\prod_{k=1}^q \(-1/\alpha+b_k\)}
		=\frac{\poch{\underline{a}}{\lambda^2}}{\poch{\underline{b}}{\lambda^2}}.
	\end{align*}
\end{example}
\begin{example}\label{ex:stability}
	In this example, we illustrate why the stability condition in \cref{thm:B} is required.
	For an alternative condition, see \cite[Theorem~2.11]{Yan92} and the discussion below.

	Let $n=2$, and $p=q=0$. By \cref{eqn:0F0}, we have
	\begin{align*}
		F(x_1,x_2) = \hyper{0}{0}(\text{---};\text{---};x_1,x_2;\alpha) = \exp(x_1+x_2).
	\end{align*}
	By definitions and \cref{ex:G}, the corresponding operators are
	\begin{align*}
		\hyper[\L]{}{0}^{(m)} = \sum_{i=1}^m\partial_i,\quad
		\hyper[\M]{0}{}^{(m)} = m,	\quad m=1,2.
	\end{align*}
	One can readily see that $F(x_1,x_2)$ satisfies \cref{eqn:pDq1x(F)-m} for $m=1,2$:
	\begin{align}
		(\partial_1-1)F(x_1,0)&=0,	\label{eqn:0D0(F)-1}\\
		(\partial_1+\partial_2-2)F(x_1,x_2)&=0.	\label{eqn:0D0(F)-2}
	\end{align}
	and the initial condition that $F(0,0)=1$.

	It is not hard to see that in $\mathscr F^{(x_1,x_2)}$, any solution of \cref{eqn:0D0(F)-2} satisfying the initial condition has the form
	\begin{align*}
		G(x_1,x_2) = \exp(x_1+x_2) H(x_1-x_2),
	\end{align*}
	where $H(z)$ is any even formal power series with $H(0)=1$.
	Further imposing \cref{eqn:0D0(F)-1} determines that $H\equiv1$ and thus $F(x_1,x_2)=\exp(x_1+x_2)$ is the unique solution of \cref{eqn:0D0(F)-1,eqn:0D0(F)-2}.
\end{example}

Let us review the proof above.
The lowering operator $\L$ is used to \emph{lower} the degree of each $J_\lambda^*$.
After interchanging the order of summation, this produces an \emph{upper sum} (a sum over larger partitions) in the coefficients.
To compensate for this, we introduced the eigen-operator $\M$ and derived \cref{eqn:rec-up}.
However, \cref{eqn:rec-up} proceeds in the wrong direction, creating an under-determined system.
To resolve this, we used the stability condition to double the number of equations, and then showed that the resulting over-determined system nevertheless has a unique solution.

A similar difficulty arose in Yan's study of differential equations for $\hyper{2}{1}$.
There, the resolution was to assume that the differential equation \eqref{eqn:pDq1x(F)} holds for all $n\geq2$ \cite[Theorem~2.11]{Yan92}.

An alternative approach is to avoid the difficulty altogether.
Instead of lowering, one can use the raising operator $\R$, which \emph{raises} the degree of $J_\lambda^*$.
After interchanging the summation order, the coefficients now involve a \emph{lower sum} (over smaller partitions).
Introducing another eigen-operator $\N$ then yields a simple recursion for $C_\lambda(\alpha)$.
This approach is developed in the next subsection.

\subsection{The general case: the raising operator}
Consider the raising operator $\hyper[\R]{p}{}(\underline{a})$ defined by \cref{eqn:raise-x}.
Then
\begin{align*}
	\hyper[\R]{p}{}(\underline a)\bigl(\hyper{p}{q}(x)\bigr)
	&=\sum_\mu\sum_{\lambda\cover\mu}
	\frac{\poch{\underline a}{\lambda}}{\poch{\underline b}{\mu}}
	\alpha^{|\lambda|}\binom{\lambda}{\mu}J_\lambda^*(x)
	=\sum_\lambda\frac{\poch{\underline a}{\lambda}}{\poch{\underline b}{\lambda}}
	\alpha^{|\lambda|}
	\biggl(\sum_{\mu\coveredby\lambda}
	\frac{\poch{\underline b}{\lambda}}{\poch{\underline b}{\mu}}
	\binom{\lambda}{\mu}\biggr)J_\lambda^*(x).
\end{align*}
Now, for each $r\geq0$, we need to find an eigen-operator $\mathcal H_r$, such that
\begin{align}
	\mathcal H_{r}(J_\lambda) = H_{r}(\lambda) \cdot J_\lambda,
\end{align}
where the eigenvalue $H_{r}(\lambda)$ is
\begin{align}\label{eqn:H_r-def}
	H_{r}(\lambda) \coloneqq \sum_{\mu\coveredby\lambda} \rho(\lambda/\mu)^r\binom{\lambda}{\mu}.
\end{align}
\subsubsection{The function $H_r(\lambda)$}
Consider the generating functions
\begin{align}
	\mathcal H(s) \coloneqq \sum_{r=0}^\infty \mathcal H_{r}s^r,\text{\quad and\quad}
	H(\lambda;s) \coloneqq \sum_{r=0}^\infty H_{r}(\lambda)s^r.\label{eqn:H-def}
\end{align}
In this subsection only, let
\begin{align}
	w_i=w_i(\lambda)\coloneqq\lambda_i-(i-1)/\alpha,	\quad 1\leq i\leq n,
\end{align}
which is equal to $\rho(\lambda/(\lambda-\varepsilon_i))+1$ if $\lambda-\varepsilon_i$ is a partition.
(Note that previously we used $\mu$ instead of $\lambda$.)
\begin{lemma}\label{lem:bino-compute}
	Let $\mu=\lambda-\varepsilon_{i_0}$ be a partition for some $i_0$, then
	\begin{align}
		\binom{\lambda}{\mu}
		&=	\(w_{i_0}+(n-1)/\alpha\)\cdot \prod_{\substack{i=1\\i\neq i_0}}^{n}\frac{w_i-w_{i_0}+1/\alpha}{w_i-w_{i_0}}
		\\&=	\(\lambda_{i_0}+(n-i_0)/\alpha\)\cdot \prod_{\substack{i=1\\i\neq i_0}}^{n}\frac{\lambda_{i_0}-\lambda_i+(i-i_0-1)/\alpha}{\lambda_{i_0}-\lambda_i+(i-i_0)/\alpha}.\notag
	\end{align}
\end{lemma}
\begin{proof}
	Let $R$ and $C$ be as in \cref{prop:bino}, then
	\begin{align*}
		\binom{\lambda}{\mu} = \prod_{s\in C}\frac{c_\lambda(s)}{c_\mu(s)} \prod_{s\in R}\frac{c_\lambda'(s)}{c_\mu'(s)}.
	\end{align*}
	Since
	\begin{align}
		\prod_{s\in\text{row }i_0} c_\lambda'(s) = \alpha^{\lambda_{i_0}}\poch{1+(n-i_0)/\alpha}{\lambda_{i_0}} \prod_{i=i_0+1}^n \frac{\poch{1+(i-i_0-1)/\alpha}{\lambda_{i_0}-\lambda_i}}{\poch{1+(i-i_0)/\alpha}{\lambda_{i_0}-\lambda_i}},
	\end{align}
	we have
	\begin{align*}
		\prod_{s\in R}\frac{c_\lambda'(s)}{c_\mu'(s)}
		=\frac{\frac{1}{\alpha}\prod_{s\in\text{row }i_0}c_\lambda'(s)}{\prod_{s\in\text{row }i_0}c_\mu'(s)}
		=	\(\lambda_{i_0}+(n-i_0)/\alpha\)\cdot \prod_{i=i_0+1}^n \frac{\lambda_{i_0}-\lambda_i+(i-i_0-1)/\alpha}{\lambda_{i_0}-\lambda_i+(i-i_0)/\alpha}.
	\end{align*}
	For the column $C$, we have
	\begin{align*}
		\prod_{s\in C}\frac{c_\lambda(s)}{c_\mu(s)}
		=	\prod_{i=1}^{i_0-1}\frac{\lambda_{i_0}-\lambda_i+(i-i_0-1)/\alpha}{\lambda_{i_0}-\lambda_i+(i-i_0)/\alpha}.
	\end{align*}
	The desired identity follows immediately.
\end{proof}
Note that when $\lambda-\varepsilon_{i_0}$ is not a partition, i.e., when $\lambda_{i_0}=0$ or when $i_0<n$ and $\lambda_{i_0}=\lambda_{i_0+1}$, the expression in \cref{lem:bino-compute} is zero, since the numerator indexed by $i=i_0+1$ is 0.
\begin{prop}\label{lem:H-GF}
	The generating function $H(\lambda;s)$ is
	\begin{align}
		H(\lambda;s) = \frac{\alpha}{s^2}\(1+(1-1/\alpha)s-\(1+(1+(n-1)/\alpha)s\) \tilde H(\lambda;s)\),
	\end{align}
	where
	\begin{align}
		\tilde H(\lambda;s) \coloneqq \prod_{i=1}^n \frac{1-(w_{i}-1+1/\alpha)s}{1-(w_i-1)s}.
	\end{align}
\end{prop}
\begin{proof}
	By \cref{eqn:H_r-def,eqn:H-def} and \cref{lem:bino-compute}, we have
	\begin{align*}
		H(\lambda;s) = \sum_{\mu\coveredby\lambda} \binom{\lambda}{\mu}\frac{1}{1-\rho(\lambda/\mu)s}
		=\sum_{i_0=1}^n \frac{w_{i_0}+(n-1)/\alpha}{1-(w_{i_0}-1)s}\cdot \prod_{\substack{i=1\\i\neq i_0}}^{n}\frac{w_i-w_{i_0}+1/\alpha}{w_i-w_{i_0}}.
	\end{align*}
	We first perform the following partial fraction calculation with independent $w_i$ and then specialize the resulting rational identity to $w_i(\lambda)$
	\begin{align*}
		-\alpha\frac{1+(1+(n-1)/\alpha)s}{s^2} \tilde H(\lambda;s)
		=	\sum_{i_0=1}^n \frac{A_{i_0}}{1-(w_{i_0}-1)s}+\frac{a_0+a_1s}{s^2}.
	\end{align*}
	For $i_0=1,\dots,n$, multiply by $1-(w_{i_0}-1)s$ and evaluate at $s=\frac{1}{w_{i_0}-1}$ (note that $w_i$'s are pairwise distinct), then we have
	\begin{align*}
		A_{i_0} =	\(w_{i_0}+(n-1)/\alpha\) \prod_{i\neq i_0} \frac{w_{i_0}-w_{i}-1/\alpha}{w_{i_0}-w_i}.
	\end{align*}
	Also, it is not hard to find that
	\begin{align*}
		\frac{a_0+a_1s}{s^2}=-\alpha\frac{1+(1-1/\alpha)s}{s^2}.
	\end{align*}
	Then we have
	\begin{align*}
		H(\lambda;s) = \sum_{i_0=1}^n \frac{A_{i_0}}{1-(w_{i_0}-1)s} = \alpha\frac{1+(1-1/\alpha)s}{s^2}-\alpha\frac{1+(1+(n-1)/\alpha)s}{s^2} \tilde H(\lambda;s).
	\end{align*}
\end{proof}
\subsubsection{The operators $\mathcal H_r$ and $\hyper[\N]{}{q}(\underline{b})$}
Using the Debiard--Sekiguchi operator $D(t)$ from \cref{eqn:Sekiguchi}, we give the generating operator $\mathcal H(s)$ explicitly:
\begin{align*}
	\mathcal H(s)
	&\coloneqq \frac{\alpha+(\alpha-1)s}{s^2}
	-\frac{\alpha+(\alpha+n-1)s}{s^2}
	\frac{D(-1/s-1+1/\alpha)}{D(-1/s-1)}.
\end{align*}

Define $\tilde D_r$ by
\begin{align}
	D(t) = \sum_{r=0}^n (t+1)^{n-r} \tilde D_r,
\end{align}
then $\tilde D_r$ and $D_r$ (see \cref{ex:D}) are related by
\begin{align}
	\tilde D_r = \sum_{k=0}^r (-1)^{r-k}\binom{n-k}{r-k}D_k.
\end{align}
Now define the operator $\tilde{\mathcal H}(s)$ and expand it as a Taylor series (in $s$ near $s=0$) as
\begin{align}
	\tilde{\mathcal H}(s)
	&\coloneqq \frac{D(-1/s-1+1/\alpha)}{D(-1/s-1)}
	=	\frac{\sum_{r=0}^n (-s)^{r}(1-s/\alpha)^{n-r}\tilde D_r}{\sum_{r=0}^n (-s)^{r}\tilde D_r}\notag
	\\&=	\sum_{r=0}^n (-s)^{r}(1-s/\alpha)^{n-r}\tilde D_r \cdot \sum_{k=0}^{\infty}\(-\sum_{r=1}^n (-s)^{r}\tilde D_r\)^k.	\label{eqn:tildeH}
\end{align}
The quotient is expanded formally in $s$ after clearing the common power $s^{-n}$.
It begins with $1-ns/\alpha$, so the apparent pole in $\mathcal H(s)$ at $s=0$ cancels.
The eigenvalue formula for $D(t)$ gives
\begin{align}
	\tilde{\mathcal H}(s)(J_\lambda) = \tilde H(\lambda;s)\cdot J_\lambda.
\end{align}
Consequently, by \cref{lem:H-GF}, we have $\mathcal H(s)(J_\lambda)=H(\lambda;s)\cdot J_\lambda$.
Define $\tilde{H}_r(\lambda)$ and $\tilde{\mathcal H}_r$ by
\begin{align}
	\tilde{H}(\lambda;s)\eqqcolon	\sum_{r=0}^\infty \tilde{H}_r(\lambda) s^r\text{\quad and \quad}
	\tilde{\mathcal H}(s)\eqqcolon	\sum_{r=0}^\infty \tilde{\mathcal H}_r s^r,
\end{align}
then
\begin{align}
	\tilde{\mathcal H}_r(J_\lambda) = \tilde H_r(\lambda)\cdot J_\lambda.
\end{align}
By \cref{lem:H-GF}, for $r\geq0$, $H_r$ and $\tilde H_r$, and hence $\mathcal H_r$ and $\tilde {\mathcal H}_r$ are related by
\begin{align}
	H_r = -\alpha \tilde H_{r+2}-(\alpha+n-1)\tilde H_{r+1},	\text{\quad and \quad}
	\mathcal H_r = -\alpha \tilde{\mathcal H}_{r+2}-(\alpha+n-1)\tilde{\mathcal H}_{r+1}.\label{eqn:H-rec}
\end{align}

\begin{example}\label{ex:H}
	The first few  operators $\tilde D_r$, $\tilde{\mathcal H}_r$ and $\mathcal H_r$ are
	\begin{align*}
		\tilde D_0&=D_0=1,\\
		\tilde D_1&=D_1-nD_0,\\
		\tilde D_2&=D_2-(n-1)D_1+\frac{n(n-1)}2D_0.
	\end{align*}
	\begin{align*}
		\tilde{\mathcal H}_0	&=	1,\\
		\tilde{\mathcal H}_1	&=	-\frac{n}{\alpha},\\
		\tilde{\mathcal H}_2	&=	-\frac{\tilde D_1}{\alpha} +\frac{1}{\alpha^2}\binom{n}{2},	\\
		\tilde{\mathcal H}_3	&=	\frac{2\tilde D_2-\tilde D_1^2}{\alpha}+\frac{n-1}{\alpha^2}\tilde D_1-\frac{1}{\alpha^3}\binom{n}{3},	\\
		\tilde{\mathcal H}_4	&=	-\frac{3\tilde D_3-3\tilde D_1\tilde D_2+\tilde D_1^3}{\alpha} +\frac{(n-1)\tilde D_1^2-(2n-3)\tilde D_2}{\alpha^2} -\frac{(n-1)(n-2)\tilde D_1}{2\alpha^3} +\frac{1}{\alpha^4}\binom{n}{4}.
	\end{align*}
	and
	\begin{align*}
		\mathcal H_0	&=	\tilde D_1+n+\frac{n(n-1)}{2\alpha} = E_2,	\\
		\mathcal H_1	&=	\tilde D_1^2-2\tilde D_2+\tilde D_1 -\frac{n(n-1)}{2\alpha} -\frac{n(n-1)(2n-1)}{6\alpha^2} = 2\square,	\\
		\mathcal H_2	&=	\tilde D_{1}^{3}-3\tilde D_{1}\tilde D_{2}+3\tilde D_{3}+\tilde D_{1}^{2}-2\tilde D_{2}
		-\frac{(n-1)\tilde D_{1}+\tilde D_{2}}{\alpha}
		\\&\=	+\frac{n(n-1)(n-2-3\tilde D_{1})}{6\alpha^2}
		+ \frac{n(n-1)(n-2)(3n-1)}{24\alpha^3}.
	\end{align*}

	One can derive these operators by considering the action of $e_1$, $[\square,e_1]$, $[\square,[\square,e_1]]$ on $\exp(p_1)$, as in \cref{lem:square-p1,lem:rho-bino}.
\end{example}

Now, let
\begin{align}\label{eqn:Dq1}
	\hyper[\N]{}{q}(\underline{b})
	\coloneqq \sum_{r=0}^q e_{q-r}(\underline b)\mathcal H_r
	= \CT\(\prod_{k=1}^q (b_k+s^{-1})\cdot \mathcal H(s)\),
\end{align}
where $\CT$ denotes the constant term of the Laurent series in $s$ near $s=0$.
Then $\hyper[\N]{}{q}(\underline{b})$ acts diagonally on $(J_\lambda)$ as
\begin{align}
	\hyper[\N]{}{q}(\underline{b})(J_\lambda) = \sum_{r=0}^qe_{q-r}(\underline{b})H_r(\lambda) \cdot J_\lambda= \sum_{\mu\coveredby\lambda} \frac{\poch{\underline{b}}{\lambda}}{\poch{\underline{b}}{\mu}}\binom{\lambda}{\mu}\cdot J_\lambda.
\end{align}
\subsubsection{\texorpdfstring{\cref{thm:C}}{Theorem C}}
Consider the differential operator
\begin{align}\label{eqn:pDqx1}
	\hyper[\N]{}{q}(\underline{b}) - \hyper[\R]{p}{}(\underline{a}),
\end{align}
where $\hyper[\R]{p}{}(\underline{a})$ and $\hyper[\N]{}{q}(\underline{b})$ are given by \cref{eqn:raise-x,eqn:Dq1}, respectively.

\begin{maintheorem}\label{thm:C}
	The hypergeometric series $\hyper{p}{q}(\underline{a};\underline{b};x;\alpha)$ is the unique solution in $\mathscr F^{(x)}$ of the equation
	\begin{align}\label{eqn:pDqx1(F)}
		\(\hyper[\N]{}{q}(\underline{b})-\hyper[\R]{p}{}(\underline{a})\) (F(x))=0,
	\end{align}
	subject to the initial condition that $F(\bm0_n)=1$, i.e., $C_{(0)}(\alpha)=1$.
\end{maintheorem}
\begin{proof}
	Let $F(x)$ be as in \cref{eqn:F(x)} and suppose that it satisfies \cref{eqn:pDqx1(F)}.
	Then
	\begin{align*}
		\hyper[\R]{p}{}(\underline a)(F)
		&=\sum_\lambda\alpha^{|\lambda|}
		\biggl(\sum_{\mu\coveredby\lambda}C_\mu(\alpha)
		\frac{\poch{\underline a}{\lambda}}{\poch{\underline a}{\mu}}
		\binom{\lambda}{\mu}\biggr)J_\lambda^*(x),\\
		\hyper[\N]{}{q}(\underline b)(F)
		&=\sum_\lambda C_\lambda(\alpha)\alpha^{|\lambda|}
		\biggl(\sum_{\mu\coveredby\lambda}
		\frac{\poch{\underline b}{\lambda}}{\poch{\underline b}{\mu}}
		\binom{\lambda}{\mu}\biggr)J_\lambda^*(x).
	\end{align*}
	Hence, comparing the coefficients of $J_\lambda^*(x)$ gives a recursion
	\begin{align}\label{eqn:C-rec-1}
		C_\lambda(\alpha)\sum_{\mu\coveredby\lambda} \frac{\poch{\underline{b}}{\lambda}}{\poch{\underline{b}}{\mu}}\binom{\lambda}{\mu}
		=\sum_{\mu\coveredby\lambda} C_\mu(\alpha) \frac{\poch{\underline{a}}{\lambda}}{\poch{\underline{a}}{\mu}}\binom{\lambda}{\mu},
	\end{align}
	which determines $(C_\lambda(\alpha))$ uniquely.
	For $\lambda\neq(0)$, the coefficient of $b_1\cdots b_q$ in the factor multiplying $C_\lambda(\alpha)$ is $\sum_{\mu\coveredby\lambda}\binom{\lambda}{\mu}=|\lambda|$.
	Thus this factor is a nonzero polynomial and is invertible in $\mathbb K$.

	Again, it is obvious by definitions and \cref{eqn:poch-ratio} that $C_\lambda(\alpha)=\frac{\poch{\underline{a}}{\lambda}}{\poch{\underline{b}}{\lambda}}$ satisfies the recursion \eqref{eqn:C-rec-1} and the initial condition, hence $\hyper{p}{q}(\underline{a};\underline{b};x;\alpha)$ is indeed a solution.
\end{proof}

Continuing the discussion from the end of the previous subsection, let us compare with the proof of \cref{thm:B}.
Replacing the lowering operator $\L$ by the raising operator $\R$ and the eigen-operator $\M$ by $\N$, the relation \eqref{eqn:rec-up} transforms into the recursion \eqref{eqn:C-rec-1}, and uniqueness follows immediately.
\begin{example}
	For $\hyper{2}{1}(a,b;c;x;\alpha)$, we have
	\begin{align*}
		\hyper[\N]{}{1}(c) &= c\mathcal H_0+\mathcal H_1 = cE_2+2\square	\\
		&=	\sum_i \(cx_i\partial_i+x_i^2\partial_i^2\)+\frac{2}{\alpha}\sum_{i\neq j}\frac{x_ix_j}{x_i-x_j}\partial_i,	\\
		\intertext{and}
		\hyper[\R]{2}{}(a,b) &= abe_1+(a+b)[\square,e_1]+[\square,[\square,e_1]]
		\\&=	\sum_{i} \(x_i^3\partial_i^2+(a+b+1)x_i^2\partial_i+abx_i\) +\frac{2}{\alpha}\sum_{i\neq j} \frac{x_i^2x_j}{x_i-x_j}\partial_i,
	\end{align*}
	hence
	\begin{align*}
		&\=	\hyper[\N]{}{1}(c)-\hyper[\R]{2}{}(a,b)	\notag
		\\&=\sum_i x_i\((1-x_i)x_i\partial_i^2+(c-(a+b+1)x_i)\partial_i-ab\)+\frac{2}{\alpha}\sum_{i\neq j}\frac{x_ix_j(1-x_i)}{x_i-x_j}\partial_i.
	\end{align*}
	In particular, when $n=1$, we have (changing the variable to $z$)
	\begin{align*}
		z\(z(1-z)\dod[2]{}{z}+(c-(a+b+1)z)\dod{}{z}-ab\),
	\end{align*}
	which is $z$ times Euler's differential operator in \cref{eqn:euler}.
\end{example}

\section{Future directions}\label{sec:future}
\subsection{\texorpdfstring{The ideals $\hyper[\mathcal D]{p}{q}$}{The ideals pDq}}\label{sec:ideal}
Let $\mathcal D^{(x)}$ be the ring of differential operators with coefficients in $\mathbb K[x_1,\dots,x_n,V(x)^{-1}]$.
Thus $\mathcal D^{(x)}$ consists of finite sums of the form
\begin{align}
	\sum_{}C_{a_1,\dots,a_n}(x)\partial_1^{a_1}\cdots \partial_n^{a_n},
\end{align}
where $a_1,\dots,a_n\geq0$ and $C_{a_1,\dots,a_n}(x)\in\mathbb K[x_1,\dots,x_n,V(x)^{-1}]$.

Similarly, let $\mathcal D^{(x,y)}$ be the ring of finite-order differential operators with coefficients in $\mathbb K[x_1,\dots,x_n,y_1,\dots,y_n,(V(x)V(y))^{-1}]$.

Let $\hyper[\mathcal D]{p}{q}$ consist of operators that annihilate $\hyper{p}{q}$, more precisely,
\begin{align}
	\hyper[\mathcal D]{p}{q}^{(x,y)} &= \hyper[\mathcal D]{p}{q}^{(x,y)}(\underline{a};\underline{b})  \coloneqq \Set{\delta\in\mathcal D^{(x,y)}}{\delta(\hyper{p}{q}(\underline{a};\underline{b};x,y;\alpha))=0},	\\
	\hyper[\mathcal D]{p}{q}^{(x)} &= \hyper[\mathcal D]{p}{q}^{(x)}(\underline{a};\underline{b}) \coloneqq  \Set{\delta\in\mathcal D^{(x)}}{\delta(\hyper{p}{q}(\underline{a};\underline{b};x;\alpha))=0}.
\end{align}
Obviously, $\hyper[\mathcal D]{p}{q}^{(x,y)}$ is a left ideal in $\mathcal D^{(x,y)}$ and $\hyper[\mathcal D]{p}{q}^{(x)}$ in $\mathcal D^{(x)}$.
In addition, $\hyper[\mathcal D]{p}{q}^{(x,y)}$ and $\hyper[\mathcal D]{p}{q}^{(x)}$ are closed under the commutator bracket $[A,B]=AB-BA$.

It would be an interesting question to study these ideals $\hyper[\mathcal D]{p}{q}^{(x,y)}$ and $\hyper[\mathcal D]{p}{q}^{(x)}$.

We note here that
\begin{enumerate}
	\item In the two-alphabet case, the differential operators $\hyper[\L]{}{q}^{(x)}-\hyper[\R]{p}{}^{(y)}$ and $\hyper[\L]{}{q}^{(y)}-\hyper[\R]{p}{}^{(x)}$, and the differences of Debiard--Sekiguchi operators $D_{r}^{(x)}-D_{r}^{(y)}$ are in $\hyper[\mathcal D]{p}{q}^{(x,y)}$.
	\item In the one-alphabet case, the differential operators $\hyper[\L]{}{q}-\hyper[\M]{p}{}$ and $\hyper[\N]{}{q}-\hyper[\R]{p}{}$ are in $\hyper[\mathcal D]{p}{q}^{(x)}$.
\end{enumerate}

\subsection{\texorpdfstring{$q$-analogs}{q-analogs}}\label{sec:basic}
\subsubsection{Basic hypergeometric series}
In 1846, Heine introduced a $q$-analog of the Gauss hypergeometric function, called the \mydef{basic hypergeometric series}:
\begin{align}
	\hyper[\phi]{2}{1}(a,b;c;z;q) = \sum_{k=0}^\infty \frac{\poch{a;q}{k}\poch{b;q}{k}}{\poch{c;q}{k}} \frac{z^k}{\poch{q;q}{k}},
\end{align}
where $\poch{a;q}{k}=(1-a)(1-aq)\cdots(1-aq^{k-1})$ is the $q$-\mydef{Pochhammer symbol}.
Since $\frac{1-q^a}{1-q}\to a$ as $q\to1$, we have
\begin{align}
	\lim_{q\to1}\hyper[\phi]{2}{1}(q^a,q^b;q^c;z;q) = \hyper{2}{1}(a,b;c;z).
\end{align}
The basic hypergeometric series also admits more parameters:
\begin{align}
	\hyper[\phi]{r}{s}(a_1,\dots,a_r;b_1,\dots,b_s;z;q) = \sum_{k=0}^\infty \((-1)^kq^{\binom{k}{2}}\)^{1+s-r} \frac{\poch{a_1;q}{k}\cdots\poch{a_r;q}{k}}{\poch{b_1;q}{k}\cdots\poch{b_s;q}{k}} \frac{z^k}{\poch{q;q}{k}}.
\end{align}
The series $\hyper[\phi]{r}{s}$ satisfies the following $q$-difference equation:
\begin{align}
	\Delta_1\Delta_{b_1/q}\cdots \Delta_{b_s/q}(F(z)) =z\Delta_{a_1}\cdots\Delta_{a_r}(F(q^{1+s-r}z)),
\end{align}
where $\Delta_af(z)=af(qz)-f(z)$.
This equation is the $q$-analog of \cref{eqn:pfq-de}.
See \cite[Exercise~1.31]{GR04}.
\subsubsection{Macdonald hypergeometric series}
In another manuscript in 1980s, \cite{Mac-HG-2}, Macdonald introduced hypergeometric series associated with Macdonald polynomials $J_\lambda(q,t)$:
\begin{align}
	\hyper[\Phi]{r}{s}(\underline{a};\underline{b};x;q,t) &= \sum_{\lambda} \frac{\poch{\underline{a};q,t}{\lambda}}{\poch{\underline{b};q,t}{\lambda}} t^{n(\lambda)} \frac{J_\lambda(x;q,t)}{j_\lambda(q,t)},	\\
	\hyper[\Phi]{r}{s}(\underline{a};\underline{b};x,y;q,t) &= \sum_{\lambda} \frac{\poch{\underline{a};q,t}{\lambda}}{\poch{\underline{b};q,t}{\lambda}} t^{n(\lambda)} \frac{J_\lambda(x;q,t)J_\lambda(y;q,t)}{j_\lambda(q,t)  J_\lambda(1,t,\dots,t^{n-1};q,t)},
\end{align}
where $\underline{a}=(a_1,\dots,a_r)$, $\underline{b}=(b_1,\dots,b_s)$.
See also \cite{Kan96} for a slightly different definition (the two definitions coincide when $r=s+1$).
Macdonald hypergeometric series generalize both the basic and the Jack hypergeometric series.

The differential-equation characterizations in this paper have $q$-difference analogs for Macdonald hypergeometric series.
The case of $\hyper[\Phi]{2}{1}(a,b;c;x;q,t)$ was studied by Kaneko \cite{Kan96} in 1996.
Chen \cite{Chen-Mac} treats the general case $\hyper[\Phi]{r}{s}$ and gives $q$-difference equations characterizing the series in one and two alphabets.

\subsection{Non-symmetric analogs}
The non-symmetric analogs of Jack polynomials, differential-reflection operators (Cherednik operators), binomial formulas, and hypergeometric series have been studied in \cite{Cherednik91,Sahi96-nonsym,BF-nonsym,BR23} and the references therein.
Our methods extend naturally to this non-symmetric setting and will be developed in another forthcoming paper \cite{CS-nsJack}.
\appendix
\section{Borodin--Olshanski's \texorpdfstring{${}_2\widehat{F}_1$}{2F̂1}}\label[appendix]{appendix}
In \cite{BO05}, Borodin and Olshanski introduced the following hypergeometric series with a different denominator:
\begin{align}
	\hyper[\widehat{F}]{2}{1}(a,b;c;x;\alpha)
	= \sum_{\lambda\in\mathcal P_n} \frac{\poch{a;\alpha}{\lambda}\poch{b;\alpha}{\lambda}}{\poch{c}{|\lambda|}} \alpha^{|\lambda|} J_\lambda^*(x;\alpha).
\end{align}
We present two differential equations characterizing $\hyper[\widehat{F}]{2}{1}$.

\subsection{The lowering operator}
Define
\begin{align}
	\widehat\L = \widehat\L^{(n)}(c) = (c+E_2)E_1 = \(c+\sum_{i=1}^n x_i\partial_i\)\sum_{i=1}^n \partial_i,
\end{align}
then
\begin{align}
	\widehat\L( J_\lambda^*) = (c+|\lambda|-1) \sum_{\mu\coveredby\lambda} \binom{\lambda}{\mu} \frac{J_\lambda^*(\bm1_n)}{J_\mu^*(\bm1_n)} J_\mu^*.
\end{align}
Hence,
\begin{align*}
	\widehat\L\bigl(\hyper[\widehat{F}]{2}{1}\bigr)
	&=\sum_{\mu\in\mathcal P_n}
	\frac{\poch{a;\alpha}{\mu}\poch{b;\alpha}{\mu}}{\poch{c}{|\mu|}}
	\alpha^{|\mu|}
	\biggl(\alpha\sum_{\lambda\cover\mu}
	\bigl(a+\rho(\lambda/\mu)\bigr)\bigl(b+\rho(\lambda/\mu)\bigr)
	\binom{\lambda}{\mu}\frac{J_\lambda^*(\bm1_n)}{J_\mu^*(\bm1_n)}\biggr)J_\mu^*.
\end{align*}
By \cref{lem:square-p1,lem:rho-bino}, let
\begin{align}
	\M&=\M^{(n)}(a,b)=2\square+\(a+b+1+\frac{n-1}{\alpha}\)E_2+abn
	\\&= \sum_{i=1}^n \(x_i^2\partial_i^2+(a+b+1)x_i\partial_i+ab\)
	+\frac{1}{\alpha} \sum_{1\leq i\neq j\leq n} \frac{x_i+x_j}{x_i-x_j}x_i\partial_i.
\end{align}
then
\begin{align}
	\M(J_\mu^*) = \sum_{\lambda\cover\mu} \alpha(a+\rho(\lambda/\mu))(b+\rho(\lambda/\mu))\binom{\lambda}{\mu} \frac{J_\lambda^*(\bm1_n)}{J_\mu^*(\bm1_n)}\cdot J_\mu^*.
\end{align}
Then by the same reasoning as in \cref{thm:B}, we have:
\begin{namedtheorem*}{Theorem $\widehat{B}$}
	The series $\hyper[\widehat{F}]{2}{1}(a,b;c;x;\alpha)$ is the unique solution in $\mathscr F^{(x)}$ of the equation
	\begin{align}
		(\widehat\L^{(n)}(c)-\M^{(n)}(a,b))(F(x))=0,\quad F(\bm0_n)=1,
	\end{align}
	subject to the following stability condition:
	for every $1\leq m<n$, the specialization $F_m$ defined by
	\begin{align*}
		F_m(x_1,\dots,x_m)\coloneqq F(x_1,\dots,x_m,0,\dots,0)
	\end{align*} satisfies
	\begin{align}
		\bigl(\widehat\L^{(m)}(c)-\M^{(m)}(a,b)\bigr)(F_m(x_1,\dots,x_m))=0.
	\end{align}
\end{namedtheorem*}
\subsection{The raising operator}
Let
\begin{align}
	\R &= \R^{(n)}(a,b) = (\ad_{\square}+a)(\ad_{\square}+b)(e_1)
	\\&=\sum_{i=1}^n \(x_i^3\partial_i^2+(a+b+1)x_i^2\partial_i+abx_i\) +\frac{2}{\alpha}\sum_{1\leq i\neq j\leq n}\frac{x_ix_j}{x_i-x_j}x_i\partial_i
\end{align}
then
\begin{align}
	\R(J_\mu^*) = \sum_{\lambda\cover\mu} \alpha (a+\rho(\lambda/\mu))(b+\rho(\lambda/\mu)) \binom{\lambda}{\mu}J_\lambda^*
\end{align}
hence
\begin{align*}
	\R\bigl(\hyper[\widehat{F}]{2}{1}\bigr)
	&=\sum_{\lambda\in\mathcal P_n}
	\frac{\poch{a;\alpha}{\lambda}\poch{b;\alpha}{\lambda}}{\poch{c}{|\lambda|}}
	\alpha^{|\lambda|}
	\biggl(\sum_{\mu\coveredby\lambda}(c+|\mu|)
	\binom{\lambda}{\mu}\biggr)J_\lambda^*.
\end{align*}
Now, by \cref{ex:H}, set
\begin{align}
	\widehat\N = \widehat\N^{(n)}(c)= (c-1+E_2)E_2 = \left(c-1+\sum_{i=1}^n x_i\partial_i\right)\sum_{i=1}^n x_i\partial_i,
\end{align}
then
\begin{align}
	\widehat\N(J_\lambda^*)
	=   (c+|\lambda|-1) \sum_{\mu\coveredby\lambda}\binom{\lambda}{\mu} \cdot J_\lambda^*
	=   \sum_{\mu\coveredby\lambda}(c+|\mu|)\binom{\lambda}{\mu} \cdot J_\lambda^*.
\end{align}
(In fact, we have $\sum_{\mu\coveredby\lambda}\binom{\lambda}{\mu} = |\lambda|$.)
Again, by the same reasoning as in \cref{thm:C}, we have
\begin{namedtheorem*}{Theorem $\widehat{C}$}
	The series $\hyper[\widehat{F}]{2}{1}(a,b;c;x;\alpha)$ is the unique solution in $\mathscr F^{(x)}$ of the equation
	\begin{align}
		(\widehat\N^{(n)}(c)-\R^{(n)}(a,b))(F(x))=0,\quad F(\bm0_n)=1.
	\end{align}
\end{namedtheorem*}

\section*{Acknowledgments}
The authors would like to thank Donald Richards for suggesting the question of finding differential operators for Macdonald's hypergeometric series.
During this work, H.C.\ was partially supported by the Lebowitz Summer Research Fellowship and the SAS Fellowship at Rutgers University.
S.S.\ was partially supported by Simons Foundation grant 00006698.
\medskip

\textbf{AI statement.} The original research for this paper was completed without the use of AI tools.
During the preparation of the first author's doctoral thesis, typographical errors were corrected without AI assistance.
In preparing this revision, we used ChatGPT-6 Astra to correct typographical errors and help revise the proof of Theorem~B.

\printbibliography
\end{document}